\documentclass[reqno,10pt,dvipsnames]{amsart}

\usepackage[a4paper,left=25mm,right=25mm,top=30mm,bottom=30mm,marginpar=25mm]{geometry} 
\usepackage{amsmath}
\usepackage{amssymb}
\usepackage{amsthm}
\usepackage{amscd}
\usepackage{dsfont}
\usepackage[ansinew]{inputenc}
\usepackage{color}
\usepackage[final]{graphicx}
\usepackage{tikz}
\usepackage{bbm}
\usepackage{comment}
\usepackage{parskip}
\usepackage{pgfplots}
\usetikzlibrary{intersections, pgfplots.fillbetween}
\pgfplotsset{every axis/.append style={                    
                    label style={font=\footnotesize},
                    tick label style={font=\tiny}  
                    }}
\usetikzlibrary{patterns} 
\usetikzlibrary{positioning}
\usetikzlibrary{calc}
\usetikzlibrary{intersections}

\usepackage{enumitem}

\usepackage{float}
\usepackage[colorlinks=true, 
 hypertexnames=false,
 pdfstartview=FitV, 
 linkcolor=olive,
 citecolor=olive,
 urlcolor=olive
]{hyperref}
\usepackage{caption}
\usepackage{mathtools}
\mathtoolsset{showonlyrefs}

\renewcommand{\epsilon}{\varepsilon}
\numberwithin{equation}{section}

\newtheoremstyle{thmlemcorr}{10pt}{10pt}{\itshape}{}{\bfseries}{.}{10pt}{{\thmname{#1}\thmnumber{ #2}\thmnote{ (#3)}}}
\newtheoremstyle{thmlemcorr*}{10pt}{10pt}{\itshape}{}{\bfseries}{.}\newline{{\thmname{#1}\thmnumber{ #2}\thmnote{ (#3)}}}
\newtheoremstyle{defi}{10pt}{10pt}{\itshape}{}{\bfseries}{.}{10pt}{{\thmname{#1}\thmnumber{ #2}\thmnote{ (#3)}}}
\newtheoremstyle{remexample}{10pt}{10pt}{}{}{\bfseries}{.}{10pt}{{\thmname{#1}\thmnumber{ #2}\thmnote{ (#3)}}}
\newtheoremstyle{ass}{10pt}{10pt}{}{}{\bfseries}{.}{10pt}{{\thmname{#1}\thmnumber{ A#2}\thmnote{ (#3)}}}

\theoremstyle{thmlemcorr}
\newtheorem{theorem}{Theorem}
\numberwithin{theorem}{section}
\newtheorem{lemma}[theorem]{Lemma}

\newtheorem{proposition}[theorem]{Proposition}
\newtheorem{problem}[theorem]{Problem}

\theoremstyle{thmlemcorr*}
\newtheorem{theorem*}{Theorem}
\newtheorem{lemma*}[theorem]{Lemma}
\newtheorem{corollary*}[theorem]{Corollary}
\newtheorem{proposition*}[theorem]{Proposition}
\newtheorem{problem*}[theorem]{Problem}
\newtheorem{conjecture*}[theorem]{Conjecture}

\theoremstyle{defi}
\newtheorem{definition}[theorem]{Definition}

\theoremstyle{remexample}
\newtheorem{remark}[theorem]{Remark}

\newcommand{\Acal}{\mathcal{A}}

\newcommand{\Lcal}{\mathcal{L}}

\newcommand{\NN}{\mathbb{N}}

\DeclareMathOperator{\Lip}{Lip}

\newcommand{\dd}{\;\mathrm{d}}

\newcommand{\N}{\mathbb{N}}
\newcommand{\R}{\mathbb{R}}
\newcommand{\RR}{\mathbb{R}}

\newcommand{\1}{\mathds{1}}

\newcommand{\loc}{\mathrm{loc}}

\newcommand{\eps}{\epsilon}

\newcommand{\average}{{\mathchoice {\kern1ex\vcenter{\hrule height.4pt
width 6pt
depth0pt} \kern-9.7pt} {\kern1ex\vcenter{\hrule height.4pt width 4.3pt
depth0pt}
\kern-7pt} {} {} }}
\newcommand{\aint}{\average\int}

\newcommand{\argmin}{\operatorname{argmin}}

\DeclareMathOperator{\interior}{int}

\allowdisplaybreaks

\newcommand\restr[2]{{\left.\kern-\nulldelimiterspace #1 \vphantom{\big|} \right|_{#2}}}

\newcommand{\mres}{\mathbin{\vrule height 1.4ex depth 0pt width
0.13ex\vrule height 0.13ex depth 0pt width 1.0ex}}
 
 \def\leq{\leqslant}
\def\geq{\geqslant}

\usepackage[scr=boondoxo]{mathalfa}
\usepackage{multicol}

\usepackage{aurical}
\usepackage[T1]{fontenc}
\usepackage{lmodern}
\usepackage[bbgreekl]{mathbbol}
\usepackage{amsfonts}

\DeclareSymbolFontAlphabet{\mathbb}{AMSb}
\DeclareSymbolFontAlphabet{\mathbbl}{bbold}

\title[Monotonicity of the jump set and jump amplitudes in one-dimensional TV denoising]{Monotonicity of the jump set and jump amplitudes in one-dimensional TV denoising}

\author[Riccardo Cristoferi]{Riccardo Cristoferi}
\address{Radboud University, IMAPP - Department of Mathematics, PO Box 9010, 6500 GL Nijmegen, The Netherlands}
\email{cristoferi@science.ru.nl}

\author[Rita Ferreira]{Rita Ferreira}
\address{King Abdullah University of Science and Technology (KAUST), CEMSE Division, Thuwal 23955-6900, Saudi Arabia}
\email{rita.ferreira@kaust.edu.sa}

\author[Irene Fonseca]{Irene Fonseca}
\address{Carnegie Mellon University, 5000 Forbes Avenue, Pittsburgh, PA 15213}
\email{fonseca@andrew.cmu.edu}

\author[Jos\'e A. Iglesias]{Jos\'e A. Iglesias}
\address{Department of Applied Mathematics, University of Twente, P.O. Box 217, 7500 AE Enschede, The Netherlands}
\email{jose.iglesias@utwente.nl}

\subjclass{26A45, 26A46, 94A12, 68U10.}
\keywords{total variation, denoising, jump set, jump amplitude, regularization parameter, taut string method}

\usepackage[normalem]{ulem}

\begin{document}


\begin{abstract}  
 \vspace{-12pt}   
We revisit the classical problem of denoising a one-dimensional scalar-valued function by minimizing the sum of an $L^2$ fidelity term and the total variation, scaled by a regularization parameter.
This study focuses on proving  that the jump set of solutions, corresponding to discontinuities or edges, as well as the amplitude of the jumps are  nonincreasing as the regularization parameter increases. Compared with previous works, our results apply to a strictly larger class of input functions, extending beyond the traditional setting of functions of bounded variation to any input in $L^\infty$ with left and right approximate limits everywhere. The proof leverages competitor constructions and convexity properties of the taut string problem, a well-known equivalent formulation of the TV model.
This monotonicity property reflects that the extent to which geometric and topological features of the original signal are preserved is consistent with the amount of smoothing desired when formulating the denoising method.
\vspace{-12pt}
\end{abstract}
\maketitle

\section{Introduction}\label{sect:intro}
The Total Variation (TV) model is a cornerstone of signal and image denoising due to its ability to reduce noise while preserving sharp edges. By minimizing a balance between the total variation and a fidelity term, typically measured in terms of an $L^2$ norm, the TV model effectively smooths flat regions without blurring important images' features, such as object boundaries, which makes it particularly suitable for applications where maintaining sharp geometric and topological characteristics is critical.

Since proposed by Rudin, Osher, and Fatemi \cite{RudOshFat92}, the TV model, also known as the ROF model, has  inspired extensive research, including theoretical studies on the regularization properties~\cite{AcaVog94,ChaDuvPeyPoo17,IglMerSch18}, efficient numerical algorithms~\cite{chambolle2004algorithm, ChaPoc11}, and extensions to anisotropic~\cite{GraLen10, LasMolMuc17} and higher-order formulations~\cite{bredies2010total}. These developments underscore the versatility and impact of the TV model in both theory and practice.

A fundamental question associated with the TV and other denoising models is the study of  the structure of their solutions' \textit{jump set}, i.e., the points where discontinuities occur corresponding to edges or abrupt changes in signals. This study is key in the theoretical  understanding of the  regularization's role in signal processing and image analysis, aligning closely with practical applications in signal processing, where preserving or controlling edge structures is paramount. 

Considerable attention has been given to the one-dimensional case, where the so-called taut string formulation offers a powerful lens to study solutions to the TV model. This is based on the  equivalence between the TV and the taut string problems (see Problems~\ref{prob:rof} and~\ref{prob:taut} below). The latter consists in the minimization of the length of the graph of a function with fixed boundary conditions and bilateral inequality constraints given by shifted versions of an input function. This  equivalence is by now well known, having been recognized in the discrete setting since \cite{DavKov01}, and  used in the discretization of signal processing problems \cite{HinEtAl03,Con13} as well as in the continuum setting \cite{Gra07}, with generalizations to more general constraints in \cite{GraObe08} and weighted total variation in \cite{BreIglMer22}. 

In this work, we further  explore the equivalence between the two aforementioned  problems to prove monotonicity properties of the jump set of the solutions to the TV model in the one-dimensional setting, which ensure that the locations of discontinuities (edges) in the denoised image align with the original image's structure. This key property helps maintaining the geometric and topological integrity of the image features, leading to more accurate and visually pleasing denoising results.

Our  purpose  is twofold. First, we address the question of how increasing the regularization parameter, \(\alpha\) (see \eqref{prob:rof} below), simplifies the solution by progressively eliminating jumps  and reducing their amplitude, thereby sharpening the correspondence between input signal features and regularized solutions. Second, we seek to generalize results beyond bounded variation spaces, as explored in \cite{Cri21}, by considering inputs that satisfy weaker conditions, such as the existence of left and right approximate limits everywhere. This enlargement of the admissible class of inputs is natural in imaging, as input data often fails to belong to $BV$ due to the noise or sampling artifacts that the denoising procedure helps remove. In fact, it has been experimentally and theoretically shown in \cite{GoMo01} that natural images not corrupted by noise, due to their intricate textures and fine-scale details, are not expected to be of bounded variation. Specifically, the total variation of natural images tends to infinity as resolution increases, as the distribution of small features follows a power law. This explains why standard TV-based denoising methods, which seek a solution in the space of bounded variation, often smooth out textures and small details. Our work demonstrates that even for these non-$BV$ inputs, for which the total variation regularization is highly active, the solutions of the denoising problem maintain a key structural property: the jump set and jump amplitudes remain monotonically ordered with respect to the regularization parameter. We refer to Remark~\ref{rem:morethanBV} for further discussion on the class of inputs that we consider.

The monotonicity of the jump set with respect to \(\alpha\) ensures the hierarchical structure of solutions, wherein solutions for smaller \(\alpha\) contain all the jump points present in solutions for larger \(\alpha\).
This relationship provides a mathematical foundation for interpreting the sparsification effects of regularization.

Precisely, we  address here the study of monotonicity properties  of the jump set  and of the amplitude of the jump of the solutions to the following one-dimensional signal denoising given by a ROF-type minimization: 

\begin{problem}[ROF]\label{prob:rof} Let \(f\in L^2(0,1)\) and \(\alpha\in (0,+\infty)\), and consider the functional \(\mathcal{I}_\alpha\colon BV(0,1)\to[0,+\infty) \) defined, for \(u\in BV(0,1) \), by
\begin{equation}
\label{eq:Irof}
\begin{aligned}
\mathcal{I}_\alpha[u]:=\frac12\Vert u - f\Vert^2_{L^2(0,1)}
+\alpha TV(u,(0,1)).
\end{aligned}
\end{equation}
Find \(u_\alpha\in BV(0,1) \) such that
\begin{equation}
\label{eq:minROF}
\begin{aligned}
\mathcal{I}_\alpha[u_\alpha] = \min_{u\in BV(0,1) } \mathcal{I}_\alpha [u].
\end{aligned}
\end{equation}
\end{problem}

By using basic properties of $BV(0,1)$ and an application of the direct method of the calculus of variations, this problem can easily be seen to have minimizers. Moreover, in \eqref{eq:Irof} the first term is strictly convex while the second is also convex, so the minimizer $u_\alpha$ is unique. Functions in $BV(0,1)$ need not be continuous, and in case the input function $f$ is discontinuous, the minimizers are typically discontinuous as well. Our main result is the monotonicity with respect to $\alpha$ of the jump set of the $u_\alpha$ in the one-dimensional setting.

\begin{theorem}\label{thm:main}
Let \(f\in L^\infty(0,1) \) be such that its left and right approximate limits
\begin{equation}
\label{eq:F'lim}
\begin{aligned}
f^-(x):=\lim_{y\to x^-} \aint_y^x f(t) \dd t\quad\hbox{ and }\quad
f^+(x):=\lim_{y\to x^+} \aint_x^y f(t) \dd t
\end{aligned}
\end{equation}
exist at every point $x\in(0,1)$.
Let \(\alpha_1\), \(\alpha_2\in (0,+\infty)\), and denote by \(u_{\alpha_1}\), \(u_{\alpha_2} \in BV(0,1)\)  the corresponding solutions to Problem~\ref{prob:rof}.
If $\alpha_1 \leq \alpha_2$, then \[J_{u_{\alpha_2}} \subseteq J_{u_{\alpha_1}} \subseteq J_f.\]
Moreover, we have for each $x\in J_{u_{\alpha_2}}$ that
\[
|u_{\alpha_1}^l(x) - u_{\alpha_1}^r(x)| \geq  |u_{\alpha_2}^l(x) - u_{\alpha_2}^r(x)|,
\]
where the left and right limits are defined in \eqref{eq:r_l_limit}.
\end{theorem}

It was pointed out in \cite{JalCha14} that, under the assumption $f \in L^\infty(0,1) \cap BV(0,1)$, the inclusion $J_{u_{\alpha_2}} \subseteq J_{u_{\alpha_1}}$ follows by combining the monotonicity of the jump set with respect to time for the TV flow of \cite[Thm.~2]{CasJalNov13} with the equivalence of varying the ROF parameter and advancing time in TV flow in the one-dimensional case proved in \cite[Prop.~4.2]{BriChaNovOrl11}. Under such equivalence, one can also obtain the monotonicity of the jump amplitudes using the `local estimate' $|Du| < |Df|$ of \cite{GiaLas19} or the estimates with respect to the jump amplitudes of $f$ contained in \cite[Thm.~3.1]{ChaLas24}. In contrast to that approach, the proof we present in this work applies to a larger class of input functions $f$ (see Remark \ref{rem:morethanBV} below) and uses only variational arguments mainly based on the construction of competitors. As such, it is plausible that an approach such as ours may pave the way to a proof in more general settings, like the analog of Problem \ref{prob:rof} for functions taking values in $\R^m$ with $m>1$.

\begin{remark}\label{rmk:heuristics}
Monotonicity of the jump set of minimizers and their inclusion in the jump set of the input function is consistent with the limits $\alpha \to 0$ and $\alpha \to +\infty$. In the first case, assuming $f \in BV \cap L^2$ and using $u=f$ in \eqref{eq:minROF} one readily obtains $u_\alpha \to f$ strongly in $L^2$. For large $\alpha$, one may prove by using convexity of \eqref{eq:Irof} and a Poincar\'e inequality (see \cite[Prop.~2.5.7]{Jal12}) that there is some $\alpha_f$ depending on $f$ such that for $\alpha \geq \alpha_f$ the minimizers $u_\alpha$ are constant and equal to the average of $f$ on $(0,1)$. Both of these observations hold true also for the higher-dimensional case and beyond scalar-valued functions.
\end{remark}

\begin{remark}\label{rem:morethanBV}
As we discuss next, the space of functions satisfying \eqref{eq:F'lim}  strictly contains the space of functions of bounded variation. This distinction is both of theoretical and practical interest, as it allows us to extend the robustness of the TV model to a broader class of inputs that are not typically considered in standard $BV$ theory, yet often arise from data acquisition processes. Moreover, as  proved in \cite{PaAr24, Ark18},  certain Besov spaces strictly contain $BV \cap L^\infty$ and consist of functions that are approximately continuous except at countably many points. Consequently, these works  show that certain Besov spaces, which are widely used in harmonic analysis and approximation theory, provide a natural setting for functions that are not in $BV$ but still possess properties like approximate continuity except at a countable set of points. This connection provides a more abstract, functional-analytic justification for our chosen class of inputs and situates our work within a broader framework beyond specific counterexamples such as the ones we exhibit next.

First of all, we note that, if $f\in BV$, then the above assumptions are satisfied. Indeed, fix $x_0\in(0,1)$ and consider the function $F:(0,1)\to\R$ defined as
\[
F(x):= \int_{x_0}^x f(t) \dd t.
\]
Then, since $f\in L^\infty(0,1)$, we get that $F$ is Lipschitz, and therefore of bounded variation. Thus, the validity of \eqref{eq:F'lim} follows from \cite[Thm~2.17]{Leo17}.

We now provide two explicit examples of a function that satisfies assumption \eqref{eq:F'lim}, but is not of bounded variation.
The first example is the function $f:(0,1)\to\R$ given by $f(x):= \sin(1/x)$. The function $f$ is not of bounded variation, but it is analytic, and thus satisfies \eqref{eq:F'lim}.

For the reader who is dissatisfied by the previous example since there is degeneracy at the boundary of the domain of definition, we present another more involved example.
Let $f:[-1,1]\to\R$ be defined as
\[
f(x):=
\left\{
\begin{array}{ll}
1 & \text{ if } x\in [n^{-1}, n^{-1}+2^{-n}], \text{ for some } n\geq 1,\\
0 & \text{ else}.
\end{array}
\right.
\]
This function does not have bounded pointwise variation (and thus, it is not of bounded variation). Indeed, it has countably many jumps of size one.
Nevertheless, it satisfies assumption \eqref{eq:F'lim}.
The only point where we need to check that the required condition is satisfied is $x_0=0$. Note that this point is in the interior of the domain of definition of the function $f$.
We claim that
\[
\lim_{x\to x_0^-} \aint_x^{x_0} f(x) \dd t = \lim_{x\to x_0^+} \aint_{x_0}^x f(x) \dd t = 0.
\]
We only verify that the first limit holds, since the latter follows by using a similar argument.
Let $x>0$, and let $n\in\N$ be such that $x\in [n^{-1}, (n-1)^{-1})$. Then,
\[
\frac{1}{x}\int_0^x f(t) \dd t \leq n \int_0^{\frac{1}{n-1}} f(t) \dd t
= n \sum_{k=n}^\infty \frac{1}{2^k}
= \frac{n}{2^{n-1}} \to0 \quad\text{as }n\to\infty.
\]
Using this example, it is possible to construct a function with countably many points that behave like the origin for the above function.
\end{remark}


\subsection{Strategy of the proof}

The rest of the paper is dedicated to the proof of Theorem \ref{thm:main}, building on intermediate results that may be of independent interest. Our arguments are hinged on the equivalence between the ROF Problem~\ref{prob:rof} and the taut string Problem~\ref{prob:taut} below, stating that the input and solution of Problem~\ref{prob:rof} are the derivatives of respectively the input and solution of Problem~\ref{prob:taut}. As already mentioned above, the taut string problem consists in the minimization of the length of the graph of a function with fixed boundary conditions and bilateral inequality constraints given by shifted versions of a primitive of the input function $f$.

Section~\ref{sect:TS} is dedicated to basic results about the taut string problem. We start by recalling the equivalence in Theorem~\ref{thm:gra07}, along with the likewise known `universal minimality' property of the taut string problem in Theorem~\ref{thm:uni-min}, stating that the length integrand can be replaced by any other strictly convex function without changing the unique minimizer. Then, we define the contact sets consisting of points of the domain in which either of the constraints of the taut string problem is active, with Lemma~\Ref{lem:geomtaut} being dedicated to first consequences of this definition. In this lemma, we recall the (also known) convexity and concavity properties of the solutions in the complement of the lower and upper contact sets respectively, and prove a straightforward bound on the amount of times the solution can switch between the two contact sets. Finally,  we prove in Proposition~\ref{prop:linfty_stability} that solutions of the taut string problem are stable in $L^\infty$ norm with respect to the regularization parameter, combining the universal minimality property with a $\Gamma$-convergence argument.

Section~\ref{sect:proof_main} tackles the proof of Theorem~\ref{thm:main}, relying on properties from the previous section. The first main step is Theorem~\ref{thm:optcond} in which precise optimality conditions at the contact sets are proved by a competitor construction which exploits the quantitative strict convexity of the graph length integrand. These state, roughly speaking, that the taut string minimization decreases the absolute slope in the minimizers with respect to the inputs. Taking advantage of these optimality conditions, the convexity/concavity properties of Lemma~\ref{lem:geomtaut} and the $L^\infty$ stability of Proposition~\ref{prop:linfty_stability},  we prove in Theorem~\ref{thm:contact_set_inclusion} that the contact sets are nonincreasing with respect to the regularization parameter. Independently of that result, in Proposition~\ref{prop:jump-inclusion} we obtain, directly from the optimality conditions, that the jump set of the solution of the ROF problem is contained in the intersection of the jump set of its input and the contact sets of the corresponding taut string problem. This last result is then used in Proposition~\ref{prop:from-contact-to-jump} to conclude that monotonicity of the contact sets implies monotonicty of the jump sets.
Finally, in Proposition \ref{prop:monot_ampl}, we show how the strategy we used to prove the monotonicity of the jump set can be used also to assert the monotonicity of the amplitude of the jump set, thereby completing the proof of Theorem~\ref{thm:main}.


\subsection{Notation and preliminaries}
The regularization term in \eqref{eq:Irof} is the total variation, defined as
\[TV(u,(0,1)) := \sup\left\{ \int_0^1 u(x) \psi'(x) \dd x :\, \psi\in C^1_c(0,1), \ |\psi(x)|\leq 1 \text{ for all }x \in (0,1)\right\},\]
with the space of functions of bounded variation being given by
\[BV(0,1):=\left\{ u \in L^1(0,1):\, TV(u,(0,1)) < + \infty \right\} \text{ with }\|u\|_{BV(0,1)} := \|u\|_{L^1(0,1)} + TV(u,(0,1)).\]
Equivalently, $BV(0,1)$ can be seen as the subset of $u \in L^1(0,1)$ whose distributional derivative $Du$ is a finite signed Radon measure, that is,
\begin{equation}\label{eq:distder} \int_0^1 \psi(x) \dd Du(x) = - \int_0^1 u(x) \psi'(x) \dd x \text{ for all }\psi \in C_c^\infty(0,1).\end{equation}
We will occasionally work with the usual Sobolev spaces $W^{k,p}(0,1)$ of functions $u \in L^p(0,1)$ whose distributional derivative $Du$ (see \eqref{eq:distder}) satisfies $Du \in L^p(0,1)$.
Since these are all defined by integration, strictly speaking they are spaces of equivalence classes of functions, up to equality a.e.~with respect to the Lebesgue measure $\mathcal{L}^1$ on $(0,1)$. As we will be dealing with pointwise properties of these functions, we need to pick specific representatives so that equality between functions is understood to be everywhere in $(0,1)$. For $u \in BV(0,1)$, we will work with so-called good representatives $\tilde u:(0,1) \to \R$ satisfying 
\begin{equation}\label{eq:pwvar}TV(u,(0,1))= \sup \left\{ \sum_{i=1}^{n-1} |\tilde u(x_{i+1})- \tilde u(x_i)| :\, n \geq 2, \ 0<x_1<\ldots<x_n < 1 \right\},\end{equation}
where the right-hand side is known as the pointwise variation of $\tilde u$.
This property implies \cite[Thm.~3.28]{AmbFusPal00} that the left and right limits
\[
\lim_{y \to x^-} \tilde u(y) \text{ and } \lim_{y \to x^+} \tilde u(y) \text{ exist for all } x \in (0,1),
\]
and that these are equal for all $x$ such that $Du(\{x\})=0$. Good representatives are not uniquely defined by the property \eqref{eq:pwvar}; some specific choices are the left-continuous, right-continuous, and precise representatives defined for $x \in (0,1)$ as
\begin{equation}\label{eq:r_l_limit}
u^\ell(x) := u(0) + Du((0,x)), \quad\quad \ u^r(x) := u(0) + Du((0,x]),
\end{equation}
and
\[
\ u^\ast(x) := \lim_{\delta \to 0^+} \aint_{x - \delta}^{x+ \delta} u(t) \dd t = \frac{u^\ell(x)+u^r(x)}{2},
\]
respectively. Note that the limit in the definition of $u^\ast(x)$ is known as the approximate limit of $u$ at $x$, and $u(0)$ is to be understood in the sense of traces, satisfying in particular that \cite[Thm.~5.7]{EvaGar15}
\[\lim_{\delta \to 0^+} \aint_0^{\delta} |u(t) - u(0)| \dd t = 0.\]
In the following, we always identify $u \in BV(0,1)$ or $u \in W^{k,p}(0,1)$ with their precise representative $u^\ast$. 

We will also work with the space $AC(0,1)$ of absolutely continuous functions, that is, $u \in AC(0,1)$ whenever for every $\eps >0$ there is some $\delta >0$ such that
\[\sum_{i=1}^n |u(b_i) - u(a_i)| < \eps \text{ for all pairwise disjoint } (a_i,b_i) \subset (0,1),\ i=1,\ldots,n, \text{ with }\sum_{i=1}^n (b_i - a_i) < \delta.\]
Functions in $AC(0,1)$ are differentiable $\mathcal{L}^1$-a.e.~\cite[Prop.~3.9]{Leo17}, and for them the identification with precise representatives implies $AC(0,1) = W^{1,1}(0,1)$ and $Du(x)=u'(x)$ for $\mathcal{L}^1$-a.e.~$x \in (0,1)$. Moreover, under the same identification, we also have
\[W^{1,\infty}(0,1) = \left\{ u:(0,1) \to \R :\, \Lip(u):=\sup_{x \neq y}\frac{|u(x)-u(y)|}{|x-y|}< + \infty \right\}, \text{ and }\Lip(u) = \|Du\|_{L^\infty(0,1)}.\]

Finally, for any locally integrable function $u \in L^1_\loc(0,1)$, one can define the jump set $J_u$ as the set of $x \in (0,1)$ for which the left and right approximate limits 
\[u^-(x):= \lim_{y \to x^-} \aint_y^x u(t) \dd t \quad\text{ and }\quad u^+(x):= \lim_{y \to x^+} \aint_x^y u(t) \dd t\]
both exist but are different. We remark that here and in the sequel, we will use this definition for $u^-$ and $u^+$, which differs from the common convention for BV functions in more than one dimension, in which these are respectively used to denote the lower and higher value at either side of a jump point. If additionally $u \in BV(0,1)$, then (the precise representative of) $u$ is also differentiable $\mathcal{L}^1$-a.e., and we have the decomposition
\[Du = (u') \mathcal{L}^1 + \big[\max(u^-, u^+)-\min(u^-, u^+)\big] \mathcal{H}^0\mres J_u + D^c u,\]
where $\mathcal{H}^0$ is the counting measure and the Cantor part $D^c u$ is singular with respect to the other two terms and vanishes on all countable subsets of $(0,1)$.


\section{The taut string problem}\label{sect:TS}

\begin{problem}[Taut String]\label{prob:taut} Let \(f\in L^2(0,1)\) and \(\alpha\in
(0,+\infty)\). Define
\begin{equation}\label{eq:defF}
\begin{aligned}
F(t):=\int_0^t f(x) \dd x \quad \text{for } t\in [0,1].
\end{aligned}
\end{equation}
Set
\[
\Acal_\alpha:=\left\{ U\in AC(0,1)\colon \sup_{x \in (0,1)} |U(x)-F(x)| \leq \alpha,\  U(0)=F(0)=0, \ U(1)=F(1) \right\},\]
and consider the functional \(\mathcal{J}\colon AC(0,1)\to[0,+\infty)
\) defined, for \(U\in AC(0,1) \), as
\begin{equation}
\label{eq:Jtaut}
\begin{aligned}
\mathcal{J}[U]:=\int_0^1  \sqrt{1+\big[U'(x)\big]^2} \dd x.
\end{aligned}
\end{equation}
Find \(U_\alpha\in \Acal_\alpha \) such that
\begin{equation}
\label{eq:minTaut}
\begin{aligned}
\mathcal{J}[U_\alpha] = \min_{U\in\Acal_\alpha } \mathcal{J}[U].
\end{aligned}
\end{equation}
\end{problem}
By \cite[Lemma~2]{Gra07}, there is a unique solution $U_\alpha$ to Problem~\ref{prob:taut} (see Figure \ref{fig:rofvstaut} for a numerical example). We now state the equivalence between the ROF Problem \ref{prob:rof} and the taut string Problem \ref{prob:taut}, which is a special case of a more general \emph{universal minimality} of the solution to the ROF Problem~\ref{prob:rof}, as addressed in Theorem~\ref{thm:uni-min} and Remark~\ref{rmk:onminimality}.

\begin{figure}[!ht]
\begin{tikzpicture}
\begin{axis}[height=0.4\textwidth,width=0.515\textwidth, clip=true,
    xmin=0,xmax=1,
    ymin=-1,ymax=1.5,
    enlargelimits={abs=1cm},
    axis lines = left,
    axis line style={=>},
    xlabel style={at={(ticklabel* cs:1)},anchor=north west},
    ylabel style={at={(ticklabel* cs:1)},anchor=south west}
    ]
\addplot[samples=100, domain=0:1, black, thick] {sin(4*pi*deg(x+0.25))+cos(pi*deg(x+0.25)/2)/2} node[anchor=south west, pos=0.39, font=\footnotesize]{$f$};
\addplot[samples=2, domain=0:0.206, color=OliveGreen, very thick] {-0.145};
\addplot[samples=20, domain=0.206:0.284, color=OliveGreen, very thick] {sin(4*pi*deg(x+0.25))+cos(pi*deg(x+0.25)/2)/2};
\addplot[samples=2, domain=0.284:0.456, color=OliveGreen, very thick] {0.747} node[anchor=north, pos=0.51, font=\footnotesize]{$u_\alpha = U'_\alpha$};
\addplot[samples=20, domain=0.456:0.545, color=OliveGreen, very thick] {sin(4*pi*deg(x+0.25))+cos(pi*deg(x+0.25)/2)/2};
\addplot[samples=2, domain=0.545:0.717, color=OliveGreen, very thick] {-0.373};
\addplot[samples=20, domain=0.717:0.79, color=OliveGreen, very thick] {sin(4*pi*deg(x+0.25))+cos(pi*deg(x+0.25)/2)/2};
\addplot[samples=2, domain=0.79:1, color=OliveGreen, very thick] {0.450};
\end{axis}
\end{tikzpicture}
\begin{tikzpicture}
\begin{axis}[height=0.4\textwidth,width=0.515\textwidth, clip=true,
    xmin=0,xmax=1,
    ymin=-0.1,ymax=0.25,
    ytick={-0.1,-0.05,...,0.2},
    yticklabel style={/pgf/number format/.cd, fixed},
    enlargelimits={abs=1cm},
    axis lines = left,
    axis line style={=>},
    xlabel style={at={(ticklabel* cs:1)},anchor=north west},
    ylabel style={at={(ticklabel* cs:1)},anchor=south west}
    ]
\addplot[samples=100, domain=0:1, black, thick] {-cos(4*pi*deg(x+0.25))/4/pi+sin(pi*deg(x+0.25)/2)/pi-0.202-0.03} node[anchor=north, pos=0.75, font=\footnotesize]{$F-\alpha$};
\addplot[samples=100, domain=0:1, black, thick] {-cos(4*pi*deg(x+0.25))/4/pi+sin(pi*deg(x+0.25)/2)/pi-0.202+0.03} node[anchor=south, pos=0.5, font=\footnotesize]{$F+\alpha$};
\addplot[samples=2, domain=0:0.206, color=OliveGreen, very thick] coordinates {(0,0) (0.2064,-0.0308)};
\addplot[samples=20, domain=0.206:0.284, color=OliveGreen, very thick] {-cos(4*pi*deg(x+0.25))/4/pi+sin(pi*deg(x+0.25)/2)/pi-0.202+0.03};
\addplot[samples=2, color=OliveGreen, very thick] coordinates {(0.2825,-0.0088) (0.4565,0.121)};
\addplot[samples=20, domain=0.456:0.545, color=OliveGreen, very thick] {-cos(4*pi*deg(x+0.25))/4/pi+sin(pi*deg(x+0.25)/2)/pi-0.202-0.03}  node[anchor=south, pos=0.51, font=\footnotesize]{$U_\alpha$};
\addplot[samples=2, color=OliveGreen, very thick] coordinates {(0.544,0.1374) (0.719,0.0723)};
\addplot[samples=20, domain=0.717:0.79, color=OliveGreen, very thick] {-cos(4*pi*deg(x+0.25))/4/pi+sin(pi*deg(x+0.25)/2)/pi-0.202+0.03};
\addplot[samples=2, color=OliveGreen, very thick] coordinates {(0.785,0.0738) (1,0.1718)};
\end{axis}
\end{tikzpicture}
\caption{ROF solution $u_\alpha$, constraints $F-\alpha$, $F+\alpha$ and taut string solution $U_\alpha$ with $f(x)=\sin(4\pi(x+1/4))+\cos(\pi/2(x+1/4))/2$ and $\alpha = 0.03$.}
\label{fig:rofvstaut}
\end{figure}
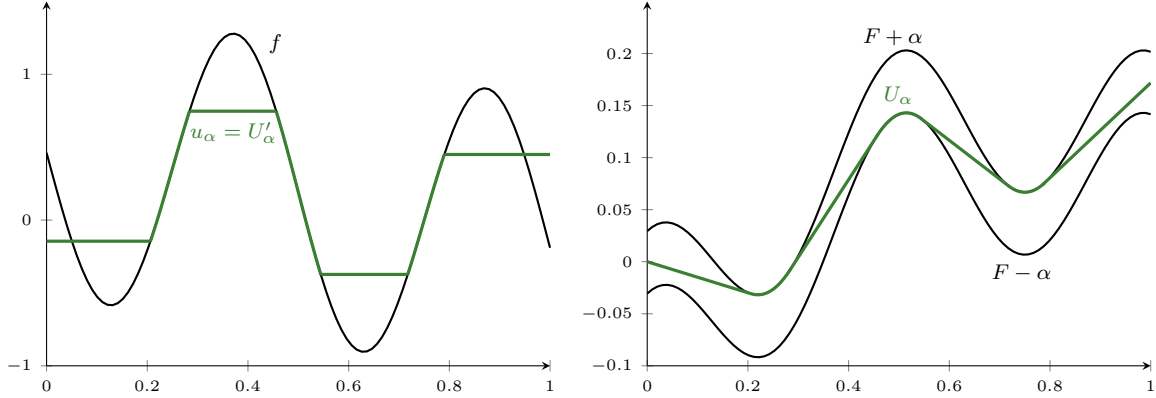

\begin{theorem}[{cf.~\cite[Proposition~2]{Gra07}}]\label{thm:gra07} Assume that \(f\in L^\infty(0,1)\), and fix \(\alpha\in(0,+\infty)\).  Let \(u_\alpha \in BV(0,1)\) be the unique solution to Problem~\ref{prob:rof} and  \(U_\alpha\in \Acal_\alpha \)  the unique solution to Problem~\ref{prob:taut}.  Then,
 \(U_\alpha'\in BV(0,1)\) and%
\begin{equation}
\label{eq:roftaut}
\begin{aligned}
u_\alpha = U_\alpha',
\end{aligned}
\end{equation}
where the equation holds for the good representatives.
\end{theorem}

As mentioned above, it turns out that minimizers of the ROF Problem~\eqref{prob:rof} are also minimizers of a large class of variational problems.
This surprising fact is proved in \cite[Theorems 4.35 and~4.46]{SchGraGroHalLen09}. Due to its importance for our strategy (and its interest), we prove it in here by detailing the proof in the aforementioned reference.

\begin{theorem}[Universal minimality condition] \label{thm:uni-min}
	Fix  
	\(\alpha\in(0,+\infty)\),  and let  \(U_\alpha\in \Acal_\alpha\)  be the
	unique
	solution to Problem~\ref{prob:taut}. Then,  \(U_\alpha\) is also the unique
	function satisfying 
	\begin{equation}\label{eq:univmin}
		\begin{aligned}
			U_\alpha \in \argmin \bigg\{ \int_0^1 \!c(v') \dd x
			\colon v\in W^{1,2}(0,1),\, \Vert v- F\Vert_\infty \leq \alpha,\  v(0)=F(0)=0,
			\ v(1)=F(1) \bigg\},
		\end{aligned}
	\end{equation}
	for any strictly convex function $c: \R \to \R$.
 \end{theorem}

\begin{proof}
	By \cite[Theorem~4.46]{SchGraGroHalLen09}, we know that \(U'_\alpha \) is  the unique
	function satisfying 
	\begin{equation}\label{eq:minuniv}
		\begin{aligned}
			\int_0^1 c\big(U_\alpha'\big) \dd x  = \min \bigg\{ \int_0^1 c(w) \dd x
			\colon w\in L^2(0,1),\, \Vert w - f\Vert_\ast \leq \alpha \bigg\},
		\end{aligned}
	\end{equation}
	where
	\begin{equation}\label{eq:defastnum}
		\begin{aligned}
			\Vert\varphi \Vert_\ast:= \sup \left\{ \int_0^1 \varphi\, \phi\dd x \colon
			\phi \in L^2(0,1), \,  TV(\phi,(0,1)) \leq 1  \right\}.
		\end{aligned}
	\end{equation}

	By setting \(\hat w(t):=\int_0^t w(s)\dd s\),   we can re-write \eqref{eq:minuniv}
	as  
	\begin{equation*}
		\begin{aligned}
			\int_0^1 c\big(U_\alpha'\big) \dd x  = \min \bigg\{ \int_0^1 c(\hat w') \dd
			x
			\colon \hat w\in W^{1,2}(0,1),\, \hat w(0)=0, \Vert\hat w' - f\Vert_\ast
			\leq \alpha \bigg\}.
		\end{aligned}
	\end{equation*}
	Thus, to conclude the proof, it suffices to show that that 
	\begin{equation}
		\label{eq:eqsets}
		\begin{aligned}
			&\left\{ v \in W^{1,2}(0,1)\colon \Vert  v- F\Vert_\infty
			\leq \alpha,\,  v(0) = F(0), \,  v(1)= F(1) \right\}\\
			&\hspace{4cm} = \left\{ \hat w \in W^{1,2}(0,1)\colon \hat w(0)=0, \Vert\hat w'
			- f\Vert_\ast
			\leq \alpha \right\}.
		\end{aligned}
	\end{equation}

	Assume that \( v \in W^{1,2}(0,1)\) is such that \( \Vert
	v- F\Vert_\infty
	\leq \alpha\), \( v(0) = F(0)\), and  \(  v(1)= F(1)\). Then, the condition
	 \(F'=f\)  \(\Lcal^1\)-a.e.~in \((0,1)\) yields
	\begin{equation*}
		\begin{aligned}
			\Vert  v' - f\Vert_\ast &= \sup \left\{ \int_0^1 ( v - F)'\, \phi\dd
			x \colon
			\phi \in L^2(0,1), \,  TV(\phi,(0,1)) \leq 1  \right\} \\
			&= \sup \left\{ \int_0^1 (F- v )\dd D\phi
			\colon
			\phi \in L^2(0,1), \,  TV(\phi,(0,1)) \leq 1  \right\}\\
			&\leq \sup \left\{ \Vert
			v- F\Vert_\infty TV(\phi,(0,1))  \colon
			\phi \in L^2(0,1), \,  TV(\phi,(0,1)) \leq 1  \right\}\\
			&\leq \alpha.
		\end{aligned}
	\end{equation*}
	Hence, recalling that \(F(0)=0\),  the set on the left-hand side of \eqref{eq:eqsets}
	is contained in the set on the right-hand side of \eqref{eq:eqsets}.
	
	Conversely, assume that \(\hat w \in W^{1,2}(0,1)\) is such that  \(\hat
	w(0)=0\) and
	\(\Vert\hat w' - f\Vert_\ast
	\leq \alpha\). Then, setting \(\phi:=\lambda\in\RR\) in \eqref{eq:defastnum}
	we get that
	\begin{equation*}
		\begin{aligned}
			\alpha \geq \lambda\int_0^1(\hat w - F)'\dd x = \lambda (\hat w(1) - F(1))
		\end{aligned}
	\end{equation*}
	for all \(\lambda \in\RR\). Thus, \(\hat w(1) = F(1)\).
	
	Now, for each \(x_0\in(0,1)\), define the test function
	%
	\begin{equation*}
		\begin{aligned}
			\phi_{x_0}(x):= \begin{cases}
				1  &\text{if } 0 < x \leq x_0,\\
				0  &\text{if } x_0 < x <1.
			\end{cases}
		\end{aligned}
	\end{equation*}
    With \(\phi:=\pm \phi_{x_0}\) as test functions in \eqref{eq:defastnum}, yields
	\begin{equation*}
		\begin{aligned}
			\alpha \geq \pm  \int_0^1 (\hat w - F)' \phi_{x_0}\dd x = \pm\int(F - \hat
			w)\dd D
			\phi_{x_0} =\pm (F-\hat w)(x_0), 
		\end{aligned}
	\end{equation*}
	where we used the identities \(\hat w(0)=0=F(0)\) and  \(\hat w(1) = F(1)\).
	Because \(x_0\)
	was arbitrarily taken, the preceding estimate shows that \(\Vert\hat w- F\Vert_\infty
	\leq \alpha\), and this concludes the proof of \eqref{eq:eqsets}.
\end{proof}

\begin{remark}\label{rmk:onminimality}
This universal minimality criterion allows, in particular, to relate Problem \ref{prob:taut} with a canonical choice of Fenchel (pre)dual of Problem \ref{prob:rof}, which results in precisely \eqref{eq:univmin} with the integrand $c(t):=t^2/2$. This fact is well-known, for a proof we refer the reader to \cite{Ove19} in the context of continuum taut string methods and \cite{BreIglMer22} allowing for weights in the total variation and fidelity term.
\end{remark}

We now introduce some sets associated with the taut string problem that will be used later in the proof of the main result.

\begin{definition}
Given  \(\alpha\in(0,+\infty)\), let \(U_\alpha\) be the unique solution to the taut string problem, Problem~\ref{prob:taut}.
Define the upper and lower contact sets as
\begin{equation}
\label{eq:contactsets}
\begin{aligned}
&C_\alpha^+:=\left\{ x\in (0,1)\colon U_\alpha(x)= F(x) + \alpha\right\},\\
&C_\alpha^-:=\left\{ x\in (0,1)\colon U_\alpha(x)= F(x) - \alpha\right\},
\end{aligned}
\end{equation}
respectively.
\end{definition}

\begin{remark}
We observe that the sets  \((0,1)\setminus(C^-_\alpha \cup C^+_\alpha)\),  \(C_\alpha^+\), and  \(C_\alpha^-\) are mutually disjoint, and their union equals the interval \((0,1)\) because \(U_\alpha\in\Acal_\alpha \).
\end{remark}

It is possible to obtain convexity and concavity properties of $U_\alpha$ in \((0,1)\setminus C^+_\alpha\) and \((0,1)\setminus C^-_\alpha\) respectively, as already done in \cite{GraObe08} and recalled in the next lemma. Moreover, we also give a basic bound (see \cite[Lemma~9]{BauMunSieWar17} for related result in the context of minimization of the number of modes via taut string minimization) on how often the solution $U_\alpha$ can switch from $C_\alpha^-$ to $C_\alpha^+$.

\begin{lemma}\label{lem:geomtaut}
Fix   \(\alpha\in(0,+\infty)\),  let \(U_\alpha\in \Acal_\alpha\)  be the unique solution to  Problem~\ref{prob:taut}.
Then,
\begin{enumerate}

\item \(U_\alpha\) is convex on each connected component of $(0,1) \setminus C_\alpha^-$;

\item  \(U_\alpha\) is concave  on each connected component of $(0,1) \setminus C_\alpha^+$;

\end{enumerate}
In particular, \(U_\alpha\) is affine  on each connected component of $(0,1)\setminus ( C^-_\alpha \cup C^+_\alpha )$. Moreover, 
\begin{enumerate}\setcounter{enumi}{2}

\item For each $x^- \in C_\alpha^-$ and $x^+ \in C_\alpha^+$, we have
\begin{equation}\label{eq:spacing}\big|x^- - x^+\big| > \frac{2\alpha}{\mathrm{Lip}(F)},\end{equation}
where $\mathrm{Lip}(F)$ is the Lipschitz constant of $F$. Consequently, 
\[U_\alpha\text{ is convex on }\left[x^+ - \frac{2\alpha}{\mathrm{Lip}(F)}, x^+ + \frac{2\alpha}{\mathrm{Lip}(F)}\right],\text{ and concave on }\left[x^- - \frac{2\alpha}{\mathrm{Lip}(F)}, x^- + \frac{2\alpha}{\mathrm{Lip}(F)}\right].\]
\end{enumerate}
\end{lemma}

\begin{proof}
\textbf{Step 1.} We prove (1) and (2). Their proof can be found in a slightly more general setting in \cite[Theorem~3.1, Lemma~5.4]{GraObe08}, but we report their argument here for completeness. For that, let $x_0 \in (0,1) \setminus C_\alpha^-$, and denote
\[ t:= U_\alpha(x_0) - F(x_0) + \alpha >0.\]
By continuity of $U_\alpha$ and $F$, there is $\delta >0$ so that 
\begin{equation}\label{eq:contUF}|U_\alpha(x) - U_\alpha(x_0)| < \frac{t}{2}, \ |F(x) - F(x_0)| < \frac{t}{2} \quad \text{for all }x \in (x_0-\delta, x_0+\delta) \subset [0,1].\end{equation} 

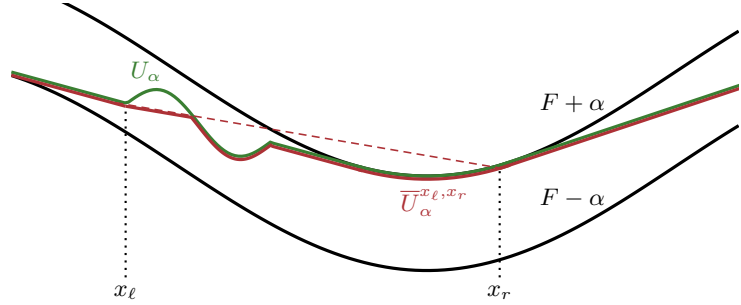
\begin{figure}[!ht]
\begin{tikzpicture}
\begin{axis}[height=0.35\textwidth,width=0.7\textwidth, clip=true,
    axis x line=none,
    axis y line=none,
    xmin=0.55,xmax=0.9,
    ymin=-0.02,ymax=0.18
    ]
\addplot[samples=100, domain=0.55:0.9, black, very thick] {-0.0022-cos(4*pi*deg(x+0.25))/4/pi+sin(pi*deg(x+0.25)/2)/pi-0.202-0.03};
\node[black, font=\small] at (axis cs:0.82,0.052){$F-\alpha$};
\addplot[samples=100, domain=0.55:0.9, black, very thick] {-cos(4*pi*deg(x+0.25))/4/pi+sin(pi*deg(x+0.25)/2)/pi-0.202+0.03};
\node[black, font=\small] at (axis cs:0.82,0.115){$F+\alpha$};
\addplot[samples=100, domain=0.55:0.719, color=OliveGreen, very thick] {0.1374+(x - 0.544)*(0.0723-0.1374)/(0.719-0.544) - (abs(x-0.64)<=0.035)*0.015*sin(29*pi*deg(x-0.64))} node[anchor=south, pos=0.35, font=\small]{$U_\alpha$};
\addplot[samples=20, domain=0.717:0.79, color=OliveGreen, very thick] {-cos(4*pi*deg(x+0.25))/4/pi+sin(pi*deg(x+0.25)/2)/pi-0.202+0.03};
\addplot[samples=2, color=OliveGreen, very thick] coordinates {(0.785,0.0738) (1,0.1718)};
\addplot[samples=2, domain=0.55:0.605, color=Maroon, very thick] {-0.0021+0.1374+(x - 0.544)*(0.0723-0.1374)/(0.719-0.544)};
\addplot[samples=2, domain=0.605:0.637, color=Maroon, very thick] {-0.00115+0.115+(x - 0.6)*(0.072-0.115)/(0.785-0.6)};
\addplot[samples=100, domain=0.6365:0.719, color=Maroon, very thick] {-0.00215+0.1374+(x - 0.544)*(0.0723-0.1374)/(0.719-0.544) - (abs(x-0.64)<=0.035)*0.015*sin(29*pi*deg(x-0.64))};
\addplot[samples=20, domain=0.717:0.79, color=Maroon, very thick] {-0.0021-cos(4*pi*deg(x+0.25))/4/pi+sin(pi*deg(x+0.25)/2)/pi-0.202+0.03}
node[anchor=north, pos=0.5, font=\small]{$\overline U^{x_\ell,x_r}_\alpha$};
\addplot[samples=2, color=Maroon, very thick] coordinates {(0.785,0.0717) (1,0.1697)};
\addplot[samples=2, color=Maroon, densely dashed, semithick] coordinates {(0.6,0.115) (0.785,0.072)};
\draw[black, thick, dotted] (axis cs:0.605, 0.0) -- (axis cs:0.605,  0.11);
\node[black, font=\small] at (axis cs:0.605,-0.01){$x_\ell$};
\draw[black, thick, dotted] (axis cs:0.785, 0.0) -- (axis cs:0.785,  0.073);
\node[black, font=\small] at (axis cs:0.785,-0.01){$x_r$};
\end{axis}
\end{tikzpicture}
\caption{Competitor construction from \cite[Lem.~5.4]{GraObe08}. If $x_\ell, x_r$ are in the same connected component of $(0,1)\setminus C_\alpha^-$ and close enough to each other, $\overline U^{x_\ell,x_r}_\alpha$ is admissible for the taut string problem and minimality of $U_\alpha$ against $\overline U^{x_\ell,x_r}_\alpha$ certifies the convexity inequality $U_\alpha\big((1-\lambda) x_\ell + \lambda x_r\big) \leq (1-\lambda) U_\alpha(x_\ell) + \lambda U_\alpha(x_r)$ for all $\lambda \in (0,1)$.}
\label{fig:convexity_competitor}
\end{figure}

We then fix two points $x_\ell,x_r\in\R$ with $x_0 - \delta < x_\ell < x_r < x_0 + \delta$ and define $\overline U^{x_\ell, x_r}_\alpha:[0,1] \to \R$ as (see Figure \ref{fig:convexity_competitor})
\begin{equation*}
	\begin{aligned}
		\overline U^{x_\ell, x_r}_\alpha (x):=
		\begin{cases}
			U_\alpha(x) & \text{if } x \in [0, x_\ell] \cup [x_r,1],\\			 
			\min\big\{U_\alpha(x), (1-\lambda) U_\alpha(x_\ell) + \lambda U_\alpha(x_r)\big\} & \text{if } x = (1-\lambda) x_\ell + \lambda x_r \text{ with }\lambda \in (0,1).
		\end{cases}
	\end{aligned}
\end{equation*}
We claim that $U_\alpha$ is convex on $(x_0 - \delta, x_0 + \delta)$.
We then conclude that $U_\alpha$ is convex on each connected component of $(0,1) \setminus C_\alpha^-$ as follows. Let $a,b$ be in a connected component of $(0,1) \setminus C_\alpha^-$.
Then, it is possible to find finitely many points $a=y^1<\dots,<y^k=b$ such that $(x^i_\ell,x^i_r)\cap (x^{i+1}_\ell,x^{i+1}_r)\neq\emptyset$, where $x^i_\ell,x^i_r$ are the points constructed above relative to $y^i$. 
For the sake of the argument, we assume $k=2$. The general case can be treated similarly.
Fix $\lambda\in(0,1)$, and let $c:=\lambda a + (1-\lambda)b$.
Then, it is possible to write
\[
c=\mu a + (1-\mu)x^1_r,\quad\quad
x^1_r = \omega x^2_\ell + (1-\omega)b,\quad\quad
x^2_\ell = \xi a + (1-\xi)x^1_r,
\]
for some $\mu,\omega,\xi\in[0,1]$.
This yields that
\[
x^1_r = \frac{\omega\xi}{1-\omega(1-\xi)}a + \frac{1-\omega}{1-\omega(1-\xi)}b.
\]
Therefore, we get
\begin{align*}
f(c) &\leq \mu f(a) + (1-\mu) f(x^1_r) \\
&\leq \mu f(a) + (1-\mu)\frac{\omega\xi}{1-\omega(1-\xi)} f(a)
        + (1-\mu)\frac{1-\omega}{1-\omega(1-\xi)} f(b).
\end{align*}
Noting that
\[
c = [\mu + (1-\mu)\frac{\omega\xi}{1-\omega(1-\xi)}] a + (1-\mu)\frac{1-\omega}{1-\omega(1-\xi)} b,
\]
we obtain the desired conclusion.

We now prove the claim.
Note that convexity of $U_\alpha$ on $(x_0 - \delta, x_0 + \delta)$ is equivalent to
\[
U_\alpha(x) \leq \overline U^{x_\ell, x_r}_\alpha (x)
\]
for all $x \in (x_\ell, x_r)$ and all $x_\ell, x_r \in (x_0-\delta, x_0+\delta)$.
To prove this inequality we assume, on the contrary, that there is some $y \in (x_\ell, x_r)$ with $U_\alpha(y) > \overline U^{x_\ell, x_r}_\alpha (y)$.
Then, again by continuity, there would be $\eta > 0$ for which
\[
U_\alpha(x) > \overline U^{x_\ell, x_r}_\alpha (x)
\]
for all $x \in (y - \eta, y+\eta)$, and
\[
U_\alpha(y - \eta) = \overline U^{x_\ell, x_r}_\alpha (y - \eta),\quad\quad\quad
U_\alpha(y + \eta) = \overline U^{x_\ell, x_r}_\alpha (y + \eta).
\]

We now claim that
\begin{equation}\label{eq:first_ineq}
\mathcal{J}\big[\overline U^{x_\ell, x_r}_\alpha\big] < \mathcal{J}[U_\alpha].
\end{equation}
Indeed, by definition, we have that
\[
\int_{[0,x_\ell]\cup[x_r,1]} \sqrt{1+[(\overline U^{x_\ell, x_r}_\alpha)']^2} \dd x
    = \int_{[0,x_\ell]\cup[x_r,1]} \sqrt{1+(U'_\alpha(x))^2} \dd x.
\]
Moreover, we can write
\begin{align*}
[x_\ell, x_r] &= \{ x = (1-\lambda) x_\ell + \lambda x_r \in [x_\ell, x_r] : U_\alpha(x) \leq (1-\lambda)U_\alpha(x_\ell)
    + \lambda U_\alpha(x_r)\} \\  &\qquad\cup\,
    \{ x = (1-\lambda) x_\ell + \lambda x_r \in [x_\ell, x_r] : U_\alpha(x) > (1-\lambda)U_\alpha(x_\ell)
    + \lambda U_\alpha(x_r)\}
    =:B\cup A,
\end{align*}
namely, as a disjoint union of two measurable sets. The set $A$ is an open set, therefore with at most countably many connected components. Thus, there exists a set of indexes $I$, which is at most countable, such that
\[
A = \bigcup_{i\in I} (a_i,b_i),
\]
for some $a_i<b_i$. Now, note that
\[
(\overline U^{x_\ell, x_r}_\alpha)'(x)=U'_\alpha(x)
\]
for almost every $x\in B$. This gives
\[
\int_B \sqrt{1+[(\overline U^{x_\ell, x_r}_\alpha)']^2} \dd x
    = \int_B \sqrt{1+(U'_\alpha(x))^2} \dd x.
\]
On the other hand, for each $i\in I$, we have that
\[
\int_{(a_i,b_i)} \sqrt{1+[(\overline U^{x_\ell, x_r}_\alpha)']^2} \dd x
    <\int_{(a_i,b_i)} \sqrt{1+(U'_\alpha(x))^2} \dd x,
\]
since $\overline U^{x_\ell, x_r}_\alpha$ is affine on $[a_i,b_i]$, 
$\overline U^{x_\ell, x_r}_\alpha(a_i)=U_\alpha(a_i)$, and $\overline U^{x_\ell, x_r}_\alpha(b_i)=U_\alpha(b_i)$.
Note that the strict inequality comes from the fact that $U_\alpha(x) > (1-\lambda)U_\alpha(x_\ell)
    + \lambda U_\alpha(x_r)$ for all $x\in(a_i,b_i)$.
This proves \eqref{eq:first_ineq}.

On the other hand, the function $\overline U^{x_\ell, x_r}_\alpha$ is admissible for \eqref{eq:minTaut}. Indeed,
\[
\overline U^{x_\ell, x_r}_\alpha \leq U_\alpha \leq F + \alpha,
\]
and using \eqref{eq:contUF} we get
\begin{align*}
& (1-\lambda) U_\alpha(x_\ell) + \lambda U_\alpha(x_r) - F\big((1-\lambda) x_\ell + \lambda x_r\big) + \alpha \\
&\hspace{2cm} > (1-\lambda) \left(U_\alpha(x_0)-\frac{t}{2}\right) + \lambda \left(U_\alpha(x_0)-\frac{t}{2}\right) - F(x_0) - \frac{t}{2} + \alpha \\ 
&\hspace{2cm} = U_\alpha(x_0) - F(x_0) + \alpha - t \geq 0.
\end{align*}
Therefore, the minimality of $U_\alpha$ yields
\[
\mathcal{J}[U_\alpha]\leq \mathcal{J}\big[\overline U^{x_\ell, x_r}_\alpha\big],
\]
giving the desired contradiction.

Analogously, we obtain that $U_\alpha$ is concave on each connected component of $(0,1) \setminus C_\alpha^+$.

\textbf{Step 2.} We prove (3). The estimate \eqref{eq:spacing} follows immediately from the definition of $C_\alpha^\pm$ and the bound 
\[\Lip(U_\alpha) \leq \Lip(F).\]
To prove the latter statement of the result, we notice that by Theorem \ref{thm:gra07}, it is equivalent to $\|u_\alpha\|_\infty \leq \|f\|_\infty$, which in turn is a well known fact that follows by truncation. Indeed, if we consider $\hat u_\alpha := \min(u_\alpha, \|f\|_\infty)$, then we have 
\[\|\hat u_\alpha - f\|_{L^2(0,1)} \leq \|u_\alpha - f\|_{L^2(0,1)} \ \text{ and }\ TV(\hat u_\alpha, (0,1)) \leq TV(u_\alpha, (0,1)),\]
with both inequalities becoming strict if we had $\|\max(u_\alpha, 0)\|_\infty > \|\max(f,0)\|_\infty$. An analogous argument applies for $\max(u_\alpha, -\|f\|_\infty)$, and combining both we get that $\|u_\alpha\|_\infty > \|f\|_\infty$ would contradict minimality of $u_\alpha$ for \eqref{eq:minROF}.
\end{proof}

As a straightforward consequence of the above universal minimality, we obtain that the solutions of the taut string Problem \ref{prob:taut} are continuous in the $L^\infty$ topology with respect to $\alpha$.

\begin{proposition}\label{prop:linfty_stability}
Let \(\{\alpha_n\}_{n\in\N}\subset(0,+\infty)\) be such that $\alpha_n\to \alpha\in(0,+\infty)$. Then,
\begin{equation}\label{eq:convUalpha}
U_{\alpha_n}\to U_\alpha
\end{equation}
with respect to the uniform convergence.
\end{proposition}

\begin{proof}
In order to prove that $U_{\alpha_n}$ converges to $U_\alpha$ with respect to the uniform convergence, we argue as follows.
For $t\geq 0$, define $\mathcal{F}_t: W^{1,2}(0,1)\to\R\cup\{+\infty\}$ as
\[
\mathcal{F}_t(v):=
\left\{
\begin{array}{ll}
\displaystyle \int_0^1 |v'|^2 \dd x & \text{ if } \Vert v- F\Vert_\infty \leq t,\  v(0)=F(0)=0, \ v(1)=F(1),\\
& \\
+\infty & \text{ else.}
\end{array}
\right.
\]
We claim that:
\begin{itemize}
\item[(i)] The sequence $\{\mathcal{F}_{\alpha_n}\}_{n\in\N}$ $\Gamma$-converges to $\mathcal{F}_\alpha$ as $n\to\infty$, with respect to the uniform convergence;
\item[(ii)] If $\{w_n\}_{n\in\N}\subset W^{1,2}(0,1)$ is such that
\[
\sup_{n\in\N}\mathcal{F}_{\alpha_n}(w_n)<\infty,
\]
then there exists a subsequence $\{w_{n_k}\}_{k\in\N}$ such that $w_{n_k}\to v$ uniformly, and $\mathcal{F}_\alpha(v)<\infty$;
\item[(iii)] It holds that
\[
\sup_{n\in\N} \mathcal{F}_{\alpha_n}(U_{\alpha_n}) < \infty.
\]
\end{itemize}
This will allow us to conclude \eqref{eq:convUalpha}. Indeed, by using standard properties of $\Gamma$-convergence \cite[Corollary~7.20]{Dal93} together with (iii) and (ii), we obtain that $U_{\alpha_n}\to v$ uniformly, where $v$ is a minimizer of $\mathcal{F}_\alpha$. By strict convexity of the energy, together with Theorem~\ref{thm:uni-min}, we know that $U_\alpha$ 
is the unique solution to the constrained minimization problem
\[
\begin{aligned}
\min \bigg\{ \int_0^1 |v'|^2 \dd x \colon v\in W^{1,2}(0,1),\, \Vert v- F\Vert_\infty \leq {\alpha_n},\  v(0)=F(0)=0, \ v(1)=F(1) \bigg\}.
\end{aligned}
\]
Thus, $U_{\alpha_n}\to U_\alpha$ uniformly, as desired.

We are thus left to prove (i), (ii), and (iii). We start with (i). We will show that
\begin{itemize}
\item[(a)] Liminf inequality: Let $\{v_n\}_{n\in\N}$ be such that $v_n\to v$ uniformly. Then,
\[
\liminf_{n\to\infty} \mathcal{F}_{\alpha_n}(v_n)\geq \mathcal{F}_\alpha(v);
\]
\item[(b)] Limsup inequality: Let $v\in W^{1,2}(0,1)$. Then, there exists $\{v_n\}_{n\in\N}$ such that $v_n\to v$ uniformly, and
\[
\limsup_{n\to\infty} \mathcal{F}_{\alpha_n}(v_n)\leq \mathcal{F}_\alpha(v).
\]
\end{itemize}
We first prove (a). Without loss of generality, we can assume that 
\[
\liminf_{n\to\infty} \mathcal{F}_{\alpha_n}(v_n)= \lim_{n\to\infty} \mathcal{F}_{\alpha_n}(v_n) < \infty \quad \hbox{ and } \quad \sup_{n\in\NN}  \mathcal{F}_{\alpha_n}(v_n) < \infty.
\]
This implies in particular that
\[
v_n(0)=F(0), \quad\quad v_n(1)=F(1),\quad\quad \Vert v_n - F\Vert_\infty \leq \alpha_n,
\]
for all $n\in\N$.
By uniform convergence, by sending $n\to\infty$, we get that
\begin{equation}\label{eq:cond_v}
v(0)=F(0), \quad\quad v(1)=F(1),\quad\quad \Vert v - F\Vert_\infty \leq \alpha.
\end{equation}
To prove the desired inequality, we argue as follows. The sequence $\{v_n\}_{n\in\N}$ is bounded in $W^{1,2}(0,1)$. Therefore, up to a not relabeled subsequence, it holds that $v_n\rightharpoonup u$ weakly in $W^{1,2}(0,1)$. Since we already know that $v_n\to v$ uniformly, we get that $v_n\rightharpoonup v$ weakly in $W^{1,2}(0,1)$. Thus, the liminf inequality follows from the lower semi-continuity of the norm with respect to weak convergence.

We now prove (b). Let $v\in W^{1,2}(0,1)$. Without loss of generality, we can assume that $\mathcal{F}_\alpha(v)<\infty$, otherwise there is nothing to prove.
For $n\in\N$, define
\[
v_n := \frac{\alpha_n}{\alpha}(v-F) + F.
\]
Note that, for each $n\in\N$, $v_n\in L^\infty(0,1)$, because $W^{1,2}(0,1) \subset L^\infty(0,1)$. Moreover, using \eqref{eq:cond_v}, it follows that 
\[
\Vert v_n - F\Vert_\infty = \frac{\alpha_n}{\alpha}\Vert v - F\Vert_\infty \leq \alpha_n,\quad\quad v_n(1) =F(1),\quad\quad v_n(0) =F(0).
\]
Also, 
\[
\Vert v_n - v\Vert_\infty \leq \left| \frac{\alpha_n}{\alpha}-1 \right| \Vert v - F\Vert_\infty
\]
and, recalling that $F'=f$  $\mathcal{L}^1$-a.e.~in $(0,1)$,
\[\Vert v'_n\Vert_{L^2(0,1)} \leq \frac{\alpha_n}{\alpha}\Vert v'\Vert_{L^2(0,1)} +\left| \frac{\alpha_n}{\alpha}-1 \right| \Vert f\Vert_{L^2(0,1)}. \]
Passing to the limit as $n\to\infty$ the two preceding estimates, the convergence $\alpha_n\to\alpha$ yields
\[
\lim_{n\to\infty} \Vert v_n - v\Vert_\infty = 0 \quad \hbox{ and } \quad \limsup_{n\to\infty} \int_0^1 |v'_n|^2 \dd x \leq \int_0^1 |v'|^2 \dd x,
\]
which concludes the proof of (b).

We now prove (iii). By Theorem~\ref{thm:gra07}, we get that
\[
\mathcal{F}_{\alpha_n}(U_{\alpha_n})
= \Vert U'_{\alpha_n}\Vert_{L^2(0,1)}
= \Vert u_{\alpha_n}\Vert_{L^2(0,1)}
\leq C\Vert u_{\alpha_n}\Vert_{BV(0,1)}.
\]
The latter are uniformly bounded by definition of the ROF problem, since $\sup_n \alpha_n<\infty$, and this concludes the proof.
\end{proof}


\section{Proof of the Monotonicity, Theorem \ref{thm:main}}\label{sect:proof_main}

We start by proving necessary optimality conditions at the contact sets in terms of left and right derivatives of $F$ and of $U$, that will be used throughout our analysis.

\begin{remark}\label{rmk:limU'}
    Assume that \(f\in L^\infty(0,1) \), fix   
\(\alpha\in(0,+\infty)\),  and let  \(U_\alpha\in \Acal_\alpha\)  be the
unique
solution to Problem~\ref{prob:taut}.   Because \(U_\alpha\in AC(0,1)\), we have for all $x_0 \in (0,1)$ that
\begin{equation*}
	\begin{aligned}
		\frac{U_\alpha(x) - U_\alpha(x_0)}{x-x_0} = \aint_x^{x_0} U'_\alpha(t)\dd t.
	\end{aligned}
\end{equation*}
Consequently,  we conclude  from the condition \(U'_\alpha\in BV(0,1)\) by Theorem~\ref{thm:gra07} that
the limits
\begin{equation}\label{eq:Urlderivatives}
    \begin{aligned}
        U'^{,+}_\alpha(x_0):=\lim_{x\to x_0^+} \frac{U_\alpha(x) - U_\alpha(x_0)}{x-x_0} \quad \hbox{and} \quad U'^{,-}_\alpha(x_0):=\lim_{x\to x_0^-} \frac{U_\alpha(x) - U_\alpha(x_0)}{x-x_0} 
    \end{aligned}
\end{equation}
exist  for all $x_0 \in (0,1)$, with 
\(U'^{,+}_\alpha(x_0) = U'^{,-}_\alpha(x_0)\) for \(\Lcal^1\)-a.e.~\(x_0\in(0,1)\).
\end{remark}

\begin{figure}[!ht]
\begin{tikzpicture}
\begin{axis}[height=0.4\textwidth,width=0.5\textwidth, clip=true,
    axis x line=none,
    axis y line=none,
    xmin=0.2,xmax=1,
    ymin=-0.075,ymax=0.45
    ]
\addplot[samples=20, domain=0.35:0.5, black, very thick] {4*x*x*x*x+0.2*(x-0.5)} ;
\addplot[samples=20, domain=0.5:1, black, very thick] {0.25-(x-0.55)*(x-0.55)+0.5*(x-0.5)} 
node[anchor=north, pos=0.6, font=\small]{$F\!-\!\alpha$};
\addplot[samples=2, color=OliveGreen, very thick] coordinates {(0.3,0.025) (0.505,0.2575)};
\addplot[samples=2, color=OliveGreen, densely dashed, semithick] coordinates {(0.505,0.2575) (0.605,0.370)}
node[anchor=south, black, pos=1, font=\small]{$U'^{,-}_\alpha(x_0)\leq f^-(x_0)$};
\addplot[samples=2, color=OliveGreen, densely dashed, semithick] coordinates {(0.355,0.1675) (0.505,0.2575)}
node[anchor=south east, pos=0.1, font=\small]{$U'^{,+}_\alpha(x_0)$};
\addplot[samples=20, domain=0.5:1, color=OliveGreen, very thick] {0.2575-(x-0.55)*(x-0.55)+0.5*(x-0.5)};
\node[color=OliveGreen, font=\small] at (axis cs:0.25,0.08){$U_\alpha$};
\node[black, font=\small] at (axis cs:0.75,-0.05){(DL)};
\draw[black, thick, dotted] (axis cs:0.5, 0.025) -- (axis cs:0.5,  0.25);
\node[black, font=\small] at (axis cs:0.5,0.0){$x_0$};
\end{axis}
\end{tikzpicture}
\hspace{1cm}
\begin{tikzpicture}
\begin{axis}[height=0.4\textwidth,width=0.5\textwidth, clip=true,
    axis x line=none,
    axis y line=none,
    xmin=0.0,xmax=0.8,
    ymin=-0.075,ymax=0.45
    ]
\addplot[samples=20, domain=0.5:0.65, black, very thick] {4*(1-x)*(1-x)*(1-x)*(1-x)+0.2*((1-x)-0.5)} ;
\addplot[samples=20, domain=0:0.5, black, very thick] {0.25-((1-x)-0.55)*((1-x)-0.55)+0.5*((1-x)-0.5)} 
node[anchor=north, pos=0.4, font=\small]{$F\!-\!\alpha$};
\addplot[samples=2, color=OliveGreen, very thick] coordinates {(0.7,0.025) (0.495,0.2575)};
\addplot[samples=2, color=OliveGreen, densely dashed, semithick] coordinates {(0.495,0.2575) (0.395,0.370)}
node[anchor=south, black, pos=1, font=\small]{$U'^{,+}_\alpha(x_0)\geq f^+(x_0)$};
\addplot[samples=2, color=OliveGreen, densely dashed, semithick] coordinates {(0.645,0.1675) (0.495,0.2575)}
node[anchor=south west, pos=0.1, font=\small]{$U'^{,-}_\alpha(x_0)$};
\addplot[samples=20, domain=0:0.5, color=OliveGreen, very thick] {0.2575-((1-x)-0.55)*((1-x)-0.55)+0.5*((1-x)-0.5)};
\node[color=OliveGreen, font=\small] at (axis cs:0.75,0.08){$U_\alpha$};
\node[black, font=\small] at (axis cs:0.25,-0.05){(DR)};
\draw[black, thick, dotted] (axis cs:0.5, 0.025) -- (axis cs:0.5,  0.25);
\node[black, font=\small] at (axis cs:0.5,0.0){$x_0$};
\end{axis}
\end{tikzpicture}
\\
\begin{tikzpicture}
\begin{axis}[height=0.35\textwidth,width=0.49\textwidth, clip=true,
    axis x line=none,
    axis y line=none,
    xmin=0.2,xmax=1,
    ymin=-0.075,ymax=0.35
    ]
\addplot[samples=20, domain=0.35:0.5, black, very thick] {0.35-(4*x*x*x*x+0.2*(x-0.5))} ;
\addplot[samples=20, domain=0.5:1, black, very thick] {0.35-(0.25-(x-0.55)*(x-0.55)+0.5*(x-0.5))}
node[anchor=south, pos=0.6, font=\small]{$F\!+\!\alpha$};
\addplot[samples=2, color=OliveGreen, very thick] coordinates {(0.3,0.325) (0.505,0.0925)};
\addplot[samples=2, color=OliveGreen, densely dashed, semithick] coordinates {(0.505,0.0925) (0.605,-0.020)};
\addplot[samples=2, color=OliveGreen, densely dashed, semithick] coordinates {(0.355,0.1825) (0.505,0.0925)};
\addplot[samples=20, domain=0.5:1, color=OliveGreen, very thick] {0.35-(0.2575-(x-0.55)*(x-0.55)+0.5*(x-0.5))};
\node[color=OliveGreen, font=\small] at (axis cs:0.3,0.25){$U_\alpha$};
\node[black, font=\small] at (axis cs:0.75,-0.05){(UL)};
\draw[black, thick, dotted] (axis cs:0.5, 0.025) -- (axis cs:0.5,  0.1);
\node[black, font=\small] at (axis cs:0.5,0.0){$x_0$};
\end{axis}
\end{tikzpicture}
\hspace{1.5cm}
\begin{tikzpicture}
\begin{axis}[height=0.35\textwidth,width=0.49\textwidth, clip=true,
    axis x line=none,
    axis y line=none,
    xmin=0.0,xmax=0.8,
    ymin=-0.075,ymax=0.35
    ]
\addplot[samples=20, domain=0.5:0.65, black, very thick] {0.35-(4*(1-x)*(1-x)*(1-x)*(1-x)+0.2*((1-x)-0.5))} ;
\addplot[samples=20, domain=0:0.5, black, very thick] {0.35-(0.25-((1-x)-0.55)*((1-x)-0.55)+0.5*((1-x)-0.5))}
node[anchor=south, pos=0.4, font=\small]{$F\!+\!\alpha$};
\addplot[samples=2, color=OliveGreen, very thick] coordinates {(0.7,0.325) (0.495,0.0925)};
\addplot[samples=2, color=OliveGreen, densely dashed, semithick] coordinates {(0.495,0.0925) (0.395,-0.020)};
\addplot[samples=2, color=OliveGreen, densely dashed, semithick] coordinates {(0.645,0.1825) (0.495,0.0925)};
\addplot[samples=20, domain=0:0.5, color=OliveGreen, very thick] {0.35-(0.2575-((1-x)-0.55)*((1-x)-0.55)+0.5*((1-x)-0.5))};
\node[color=OliveGreen, font=\small] at (axis cs:0.71,0.25){$U_\alpha$};
\node[black, font=\small] at (axis cs:0.25,-0.05){(UR)};
\draw[black, thick, dotted] (axis cs:0.5, 0.025) -- (axis cs:0.5,  0.1);
\node[black, font=\small] at (axis cs:0.5,0.0){$x_0$};
\end{axis}
\end{tikzpicture}
\caption{The different cases of the result of Theorem \ref{thm:optcond}.}
\label{fig:optcond}
\end{figure}
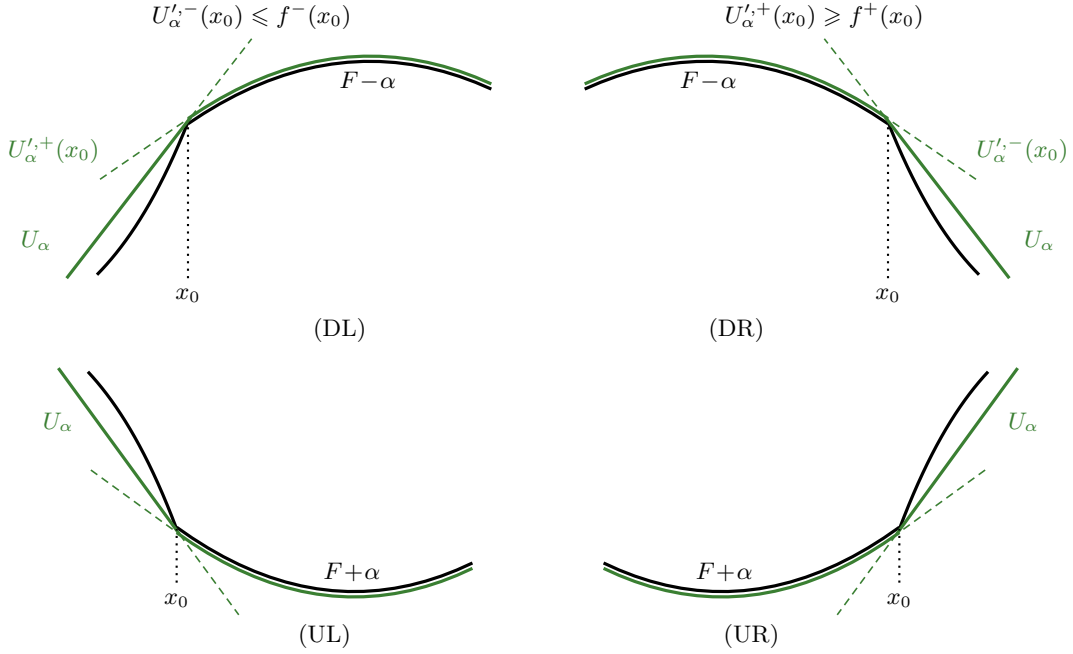

\begin{theorem}\label{thm:optcond}
Let \(f\in L^\infty(0,1) \) be such that the approximate left and right limits in \eqref{eq:F'lim} exist for all $x_0 \in (0,1)$. Fix   
\(\alpha\in(0,+\infty)\),  and let  
\(U_\alpha\in \Acal_\alpha\)  be the
unique
solution of  Problem~\ref{prob:taut}, and let $C^\pm_\alpha$ be the sets in \eqref{eq:contactsets}.
    For $x_0 \in (0,1)$, let \(U'^{,+}_\alpha(x_0)\) and \(U'^{,-}_\alpha(x_0)\) be given by \eqref{eq:Urlderivatives}. Then, the  following statements (illustrated in Figure \ref{fig:optcond}) hold for any $x_0 \in (0,1)$:
\begin{itemize}
\item[(DL)] If $x_0\in C^-_\alpha$ and there exist $\delta>0$ such that $(x_0-\delta,x_0)\subset (0,1)\setminus(C^-_\alpha \cup C^+_\alpha)$, then
\[
f^+(x_0)\leq U'^{,+}_\alpha(x_0) \leq U'^{,-}_\alpha(x_0)\leq f^-(x_0);
\]
\item[(DR)] If $x_0\in C^-_\alpha$ and there exists $\delta>0$ such that $(x_0,x_0+\delta)\subset (0,1)\setminus(C^-_\alpha \cup C^+_\alpha)$, then
\[
f^+(x_0) \leq U'^{,+}_\alpha(x_0) \leq U'^{,-}_\alpha(x_0)\leq f^-(x_0);
\]

\item[(UL)] If $x_0\in C^+_\alpha$ and there exist $\delta>0$ such that $(x_0-\delta,x_0)\subset (0,1)\setminus(C^-_\alpha \cup C^+_\alpha)$, then
\[
f^-(x_0) \leq U'^{,-}_\alpha(x_0) \leq U'^{,+}_\alpha(x_0)\leq f^+(x_0);
\]

\item[(UR)] If $x_0\in C^-_\alpha$ and there exists $\delta>0$ such that $(x_0,x_0+\delta)\subset (0,1)\setminus(C^-_\alpha \cup C^+_\alpha)$, then
\[
f^-(x_0)\leq U'^{,-}_\alpha(x_0) \leq U'^{,+}_\alpha(x_0) \leq f^+(x_0).
\]
\end{itemize}
\end{theorem}

\begin{proof}
We treat each case separately.

\textbf{Step 1.}
We  address the  (DL) case.
We first prove that $U_\alpha'^{,-}(x_0) \leq f^-(x_0)$.  
By definition of $C^-_\alpha$ and  of $\Acal_\alpha$, we get that $U_\alpha(x_0) = F(x_0) -\alpha$ and  $U_\alpha\geq F -\alpha$, respectively. Hence,
\begin{equation*}
\begin{aligned}
U_\alpha'^{,-}(x_0)= \lim_{x\to x_0^- \ } \frac{U_\alpha(x) - U_\alpha(x_0)}{x - x_0} \leq \lim_{x\to x_0^-  } \frac{F(x) - F(x_0)}{x - x_0} = f^-(x_0),
\end{aligned}
\end{equation*}
as desired.

The proof of $f^+(x_0)\leq U_\alpha'^{,+}(x_0)$ follows by using a similar argument.

We now show that $U'^{,-}_\alpha(x_0) \geq U'^{,+}_\alpha(x_0)$. For the sake of notation, we write
\[
\beta^- :=  U_\alpha'^{,-}(x_0), \quad \beta^+ :=  U_\alpha'^{,+}(x_0),
\]
and, by  contradiction, let us assume  that 
$\beta^- < \beta^+$.  
We construct a competitor $\widetilde{U}_\alpha$ by replacing part of $U_\alpha$ with a piecewise function, whose precise definition will depend on whether $\beta^+=0$ or not.

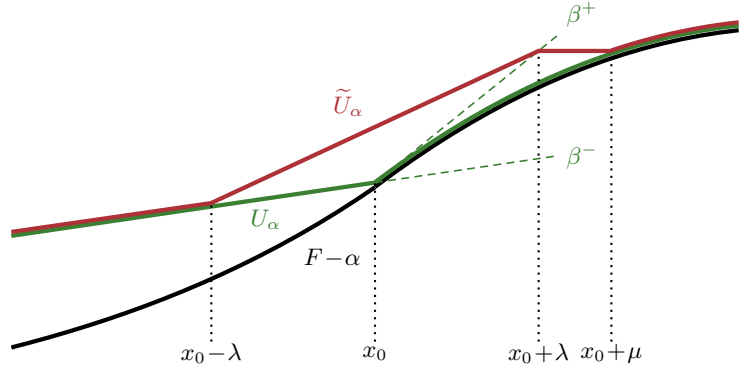
\begin{figure}[!ht]
\begin{tikzpicture}
\begin{axis}[height=0.4\textwidth,width=0.7\textwidth, clip=true,
    axis x line=none,
    axis y line=none,
    xmin=0.4,xmax=0.6,
    ymin=-0.015,ymax=0.45
    ]
\addplot[samples=20, domain=0.4:0.5, black, ultra thick] {16*x*x*x*x*x*x+0.2*(x-0.5)-0.038}
node[anchor=north west, pos=0.7, font=\small]{$F\!-\!\alpha$};
\addplot[samples=20, domain=0.5:0.6, black, ultra thick] {0.25-15*(x-0.55)*(x-0.55)+2*(x-0.5)};
\addplot[samples=2, color=OliveGreen, ultra thick] coordinates {(0.4,0.15) (0.5,0.218)}
node[anchor=north, pos=0.7, font=\small]{$U_\alpha$};
\addplot[samples=2, color=OliveGreen, densely dashed, semithick] coordinates {(0.5,0.218) (0.55,0.405)}
node[anchor=south west, pos=1, font=\small]{$\beta^+$};
\addplot[samples=2, color=OliveGreen, densely dashed, semithick] coordinates {(0.5,0.218) (0.55,0.252)}
node[anchor=west, pos=1, font=\small]{$\beta^-$};
\addplot[samples=20, domain=0.5:0.6, color=OliveGreen, ultra thick] {0.256-15*(x-0.55)*(x-0.55)+2*(x-0.5)};
\addplot[samples=2, color=Maroon, ultra thick] coordinates {(0.4,0.155) (0.4552,0.192)};
\addplot[samples=2, color=Maroon, ultra thick] coordinates {(0.455,0.192) (0.545,0.386)}
node[anchor=south east, pos=0.5, font=\small]{$\widetilde U_\alpha$};
\addplot[samples=2, color=Maroon, ultra thick] coordinates {(0.5445,0.386) (0.5652,0.386)};
\addplot[samples=20, domain=0.565:0.6, color=Maroon, ultra thick] {0.260-15*(x-0.55)*(x-0.55)+2*(x-0.5)};
\draw[black, thick, dotted] (axis cs:0.5, 0.015) -- (axis cs:0.5,  0.218);
\node[black, font=\small] at (axis cs:0.5,0.0){$x_0$};
\draw[black, thick, dotted] (axis cs:0.545, 0.015) -- (axis cs:0.545,  0.38);
\node[black, font=\small] at (axis cs:0.545,0.0){$x_0\!+\!\lambda$};
\draw[black, thick, dotted] (axis cs:0.565, 0.015) -- (axis cs:0.565,  0.38);
\node[black, font=\small] at (axis cs:0.565,0.0){$x_0\!+\!\mu$};
\draw[black, thick, dotted] (axis cs:0.455, 0.015) -- (axis cs:0.455,  0.19);
\node[black, font=\small] at (axis cs:0.455,0.0){$x_0\!-\!\lambda$};
\end{axis}
\end{tikzpicture}
\caption{Construction of the competitor $\widetilde U_\alpha$ for $\beta^+>0$, exploiting the inequality $\beta^- < \beta^+$ to obtain $\mathcal{J}[\widetilde U_\alpha] < \mathcal{J}[U_\alpha]$, a contradiction.}
\label{fig:competitorconstruction}
\end{figure}

We first consider the  $\beta^+>0$ case.
Let $0 < \mu < 2\alpha/\Lip(F)$ and $0<\lambda<\delta$, with $\lambda \leq \mu$. We define $\widetilde{U}_\alpha:[0,1]\to\R$ as (see Figure \ref{fig:competitorconstruction})
\begin{equation*}
	\begin{aligned}
		\widetilde U_\alpha (x):=
		\begin{cases}
			U_\alpha(x) & \text{if } x < x_0 - \lambda,\\
			U_\alpha(x_0 - \lambda) + \frac{1}{2} (\beta^+ + \beta^-)\big(x-(x_0-\lambda)\big)  & \text{if } x_0-\lambda \leq x < x_0 + \lambda, \\
			U_\alpha(x_0 - \lambda) +(\beta^+ + \beta^-)\lambda & \text{if } x_0 + \lambda \leq x \leq x_0 + \mu, \\
			U_\alpha(x) & \text{if } x > x_0 + \mu.
		\end{cases}
	\end{aligned}
\end{equation*}
We claim that it is possible to choose $\lambda$ and $\mu$ such that the following conditions hold:
\begin{itemize}
	\item[(i)] $U_\alpha\in AC(0,1)$;
	\item[(ii)] $U_\alpha$ is an admissible competitor for the minimization problem \eqref{eq:minTaut};
	\item[(iii)] $\mathcal{J}[\widetilde U_\alpha] < \mathcal{J}[U_\alpha]$.
\end{itemize}
This will provide the desired contradiction, since we constructed an admissible competitor whose energy is strictly lower than that of the minimizer.

To prove this energy inequality, we will use the following strict quantitative convexity of the function $t \mapsto \sqrt{1+t^2}$:  There exist $\eps_0 >0$ and $s_0>0$ such that whenever $|\xi - \beta^+| < \eps_0$ and $s<s_0$, it holds that
\begin{equation}\label{eq:strictconv}
	s + \sqrt{1+\left(\frac12 \beta^- + \frac12 \beta^+\right)^2} < \frac12 \sqrt{1+\left(\beta^-\right)^2} + \frac12 \sqrt{1+\xi^2}.
\end{equation}

Next, we observe that Lemma~\ref{lem:geomtaut} and the assumptions on $x_0$ yield that  that $U_\alpha$ is affine on $(x_0-\delta, x_0]$ and concave on $[x_0, x_0 + 2\alpha/\Lip(F))$. In particular, we have that
\begin{align}
		& U_\alpha (y) = U_\alpha(x_0) + \beta^- (y-x_0) \enspace\hbox{for all } y\in (x_0-\delta, x_0],\label{eq:cont3a}\\
		& U_\alpha(y) \leq U_\alpha(x_0) + \beta^+(y-x_0) \enspace\hbox{for all } y\in [x_0, x_0 + 2\alpha/\Lip(F)),\label{eq:cont3b}\\
		& \beta_y:=\frac{U_\alpha( y) - U_\alpha(x_0)}{y-x_0}\uparrow \beta^+  \enspace\hbox{as } y\downarrow x_0^+,\enspace \hbox{with } y<x_0 + 2\alpha/\Lip(F). \label{eq:cont3c}
\end{align}

On the other hand, using the continuity of $U_\alpha$ and $F$, together with the identity $U_\alpha(x_0) = F(x_0) -\alpha$ and the bound $\Vert U_\alpha - F\Vert_\infty \leq \alpha$, allows us to find   $ \bar \delta >0$ such that 
\begin{equation}\label{eq:cont1}
	\begin{aligned}
		& U_\alpha(x) \leq F(x) \enspace \hbox{ and } \enspace |U_\alpha (x) - U_\alpha(x_0)| < \dfrac{\alpha}{2} \enspace \hbox{ whenever } \enspace |x-x_0|\leq \bar\delta.
	\end{aligned}
\end{equation}
 Since $\beta^+ = U_\alpha'^{,-}(x_0) >0$ and $U_\alpha$ is concave on $[x_0, x_0 + 2\alpha/\Lip(F))$, up to reducing $\bar \delta$, we can also ensure that
\begin{equation}\label{eq:cont1+}
	\begin{aligned}
		& U_\alpha \hbox{ is  increasing in  } \left[x_0, x_0+  \bar\delta\,\right],
	\end{aligned}
\end{equation}
and we restrict $\bar \delta$ to satisfy
\begin{equation}\label{eq:cont2}
	\begin{aligned}
		&  \bar\delta \leq \min \left\{ \delta, \dfrac{\alpha}{2(|\beta^+ + \beta^-| + |\beta^+ - \beta^-| + \beta^+)}, \dfrac{2\alpha}{\Lip(F)}\right\}.
	\end{aligned}
\end{equation}

Finally, recalling \eqref{eq:cont3c}, we fix $\bar y \in (x_0, x_0+\bar \delta)$ such that
\begin{equation}\label{eq:cont3d}
	\begin{aligned}
		&   0<\beta^+ - \beta_{\bar y} < \eps_0 \enspace \hbox{ and } \enspace 0<\tfrac{1}{2}\big( \tfrac{\beta^+}{\beta_{\bar y}}-1\big) < s_0,
	\end{aligned}
\end{equation}
and we set
\begin{equation}\label{eq:cont4}
	\begin{aligned}
		& \mu:=\bar y-x_0,\\
		& \lambda:=\frac{U_\alpha(\bar y) - U_\alpha(x_0)}{\beta^+}.\\
	\end{aligned}
\end{equation}

Using \eqref{eq:cont3b}, \eqref{eq:cont2},   and  \eqref{eq:cont4}, we have that 
\begin{equation}\label{eq:cont5}
	\begin{aligned}
		& 0 < \mu < \bar\delta \leq  2\alpha/\Lip(F) \enspace \hbox{ and } \enspace 0< \lambda \leq \mu < \bar \delta \leq \delta.
	\end{aligned}
\end{equation}

Next, we prove that (i) holds. By definition of $\widetilde U$, \eqref{eq:cont5}, \eqref{eq:cont3a} with $y=x_0 - \lambda$, and \eqref{eq:cont4},  we have that
\begin{equation*}
\begin{aligned}
	 \lim_{x\to (x_0 +\mu)^-} \widetilde U_\alpha(x) &= U_\alpha(x_0 - \lambda) +(\beta^+ + \beta^-)\lambda\\
	 & = U_\alpha(x_0) - \beta^- \lambda +(\beta^+ + \beta^-)\lambda = U_\alpha (\bar y) =  \lim_{x\to (x_0 +\mu)^+} \widetilde U_\alpha(x), 
\end{aligned}
\end{equation*}
from which we conclude that $\widetilde U_\alpha$ is continuous in $[0,1]$. Moreover, being absolutely continuous in $[0,1]\setminus[x_0-\lambda,x_0+\mu]$, and piecewise affine in $[x_0-\lambda,x_0+\mu]$, we obtain that $\widetilde{U}_\alpha\in AC(0,1)$.

We now show that (ii) is satisfied. Using the admissibility of $U_\alpha$ and the definition of  $\widetilde U$, we are left to prove that
\begin{equation}\label{eq:cont6}
	\begin{aligned}
		& F(x) - \alpha \leq U_\alpha(x_0 - \lambda) + \frac{1}{2} (\beta^+ + \beta^-)\big(x-(x_0-\lambda)\big)  \leq  F(x) + \alpha \quad \hbox{for all } x\in [x_0 -\lambda, x_0+\lambda]
	\end{aligned}
\end{equation}
and
\begin{equation}\label{eq:cont7}
	\begin{aligned}
		& F(x) - \alpha \leq U_\alpha(x_0 - \lambda) +(\beta^+ + \beta^-)\lambda  \leq  F(x) + \alpha \quad \hbox{for all } x\in [x_0 +\lambda, x_0+\mu].
	\end{aligned}
\end{equation}

We start by establishing the lower bound in \eqref{eq:cont6}. Using the inequality $\beta^+ > \beta^-$ and  \eqref{eq:cont3a}, we have for all $x\in [x_0 -\lambda, x_0]$ that
\begin{equation*}
	\begin{aligned}
		U_\alpha(x_0 - \lambda) + \frac{1}{2} (\beta^+ + \beta^-)\big(x-(x_0-\lambda)\big) \geq U_\alpha(x_0 - \lambda) + \beta^- \big(x-(x_0-\lambda)\big) = U_\alpha(x) \geq F(x) - \alpha,
	\end{aligned}
\end{equation*}
where we used the admissibility of $U_\alpha$ in the last estimate. On the other hand, the inequality $\beta^+ - \beta^->0$ and  \eqref{eq:cont3b} yield   for all $x\in [x_0 , x_0 +\lambda]$ that
\begin{equation*}
	\begin{aligned}
		&U_\alpha(x_0 - \lambda) + \frac{1}{2} (\beta^+ + \beta^-)\big(x-(x_0-\lambda)\big) = U_\alpha(x_0) - \beta^-\lambda  + \frac{1}{2} (\beta^+ + \beta^-)\big(x-(x_0-\lambda)\big)\\
		&\quad \geq U_\alpha(x) - \beta^+(x-x_0)   - \beta^-\lambda  + \frac{1}{2} (\beta^+ + \beta^-)\big(x-(x_0-\lambda)\big) =  U_\alpha(x) + \frac{1}{2} (\beta^+ - \beta^-)\big((x_0+\lambda)-x\big)\\
		&\quad  \geq  U_\alpha(x) \geq F(x) - \alpha.
	\end{aligned}
\end{equation*}
Thus, the lower bound in \eqref{eq:cont6} holds. To establish the upper bound, we note that for all $x\in [x_0 -\lambda, x_0+\lambda]$,  \eqref{eq:cont3a} with $y=x_0 - \lambda$, \eqref{eq:cont1}, \eqref{eq:cont2},  and \eqref{eq:cont5}  yield
\begin{equation*}
	\begin{aligned}
		&U_\alpha(x_0 - \lambda) + \frac{1}{2} (\beta^+ + \beta^-)\big(x-(x_0-\lambda)\big) = U_\alpha(x_0) - \beta^-\lambda  + \frac{1}{2} (\beta^+ + \beta^-)\big(x-(x_0-\lambda)\big)\\
		&\quad \leq U_\alpha(x) +|U_\alpha(x_0) - U_\alpha(x) |  + \frac{1}{2} |\beta^+ + \beta^-||x-x_0|+ \frac{1}{2} |\beta^+ - \beta^-|\lambda\\
		&\quad  \leq  F(x) + \alpha.
	\end{aligned}
\end{equation*}
Consequently, \eqref{eq:cont6} holds.

To prove  \eqref{eq:cont7}, we take $x\in [x_0 +\lambda, x_0+\mu]$,  and observe that the definition of $\widetilde U_\alpha$,  \eqref{eq:cont3a} with $y=x_0 - \lambda$, \eqref{eq:cont4},   \eqref{eq:cont1+}, and the admissibility of $U_\alpha$, used in this order, yield
\begin{equation*}
	\begin{aligned}
		\widetilde U_\alpha (x) &=   U_\alpha   (x_0 - \lambda) + (\beta^+ + \beta^-) \lambda = U_\alpha(x_0) + \beta^+\lambda \\ 
        & = U_\alpha(x_0 + \mu) = \max_{ [x_0 +\lambda, x_0+\mu]} U_\alpha \geq U (x)   \geq  F(x) - \alpha.
	\end{aligned}
\end{equation*}
Moreover,
\begin{equation*}
	\begin{aligned}
		&\widetilde U_\alpha (x)  = U_\alpha(x_0) + \beta^+\lambda    \leq U_\alpha(x) + |U_\alpha(x_0 ) - U_\alpha(x)| + \beta^+\lambda  \leq F(x) +\alpha,
	\end{aligned}
\end{equation*}
where we used \eqref{eq:cont1}, \eqref{eq:cont2}, and \eqref{eq:cont5} in the last estimate. Hence, also 
 \eqref{eq:cont7}  holds.

We are thus left with proving (iii).
Using the definition of $\widetilde{U}_\alpha$, we need to prove the energy inequality only in the region $[x_0-\lambda, x_0+\mu]$; namely,
\begin{equation}\label{eq:ineq_energy_competitor_1}
\int_{x_0 - \lambda}^{x_0+\mu} \sqrt{1+\left(\widetilde U_\alpha'\right)^2} \dd x
<\int_{x_0 - \lambda}^{x_0+\mu} \sqrt{1+\left(U_\alpha'\right)^2} \dd x.
\end{equation}
Now,
\begin{equation}\label{eq:ineq_energy_competitor_2}
\int_{x_0 - \lambda}^{x_0+\mu} \sqrt{1+\left(\widetilde U_\alpha'\right)^2} \dd x
= (\mu-\lambda) + 2\lambda \sqrt{1+\frac{1}{4}(\beta^- + \beta^+)^2},
\end{equation}
and
\begin{align}\label{eq:ineq_energy_competitor_3}
\int_{x_0 - \lambda}^{x_0+\mu} \sqrt{1+\left(U_\alpha'\right)^2} \dd x
&=\lambda \sqrt{1 + (\beta^-)^2} + \int_{x_0}^{x_0+\mu} \sqrt{1+\left(U_\alpha'\right)^2} \dd x \nonumber\\
&\geq \lambda \sqrt{1+(\beta^-)^2} + \mu \sqrt{1+\left(\frac{U_\alpha(x_0+\mu) - U_\alpha(x_0)}{\mu}\right)^2} \nonumber \\
&\geq \lambda \sqrt{1+(\beta^-)^2} + \lambda \sqrt{1+\left(\frac{U_\alpha(x_0+\mu) - U_\alpha(x_0)}{\mu}\right)^2},
\end{align}
where the previous to last inequality follows from the fact that $U_\alpha$ is affine on $(x_0 - \lambda, x_0)$ and concave on $(x_0, x_0 +\mu)$, while the last inequality from the fact that $\mu\geq \lambda$.
Thus, dividing \eqref{eq:ineq_energy_competitor_2} and \eqref{eq:ineq_energy_competitor_3} by $2\lambda$, it follows that \eqref{eq:ineq_energy_competitor_1} is proved if we show that
\begin{equation}\label{eq:ineq_energy_competitor_4}
\frac{\mu-\lambda}{2\lambda} + \sqrt{1+\frac{1}{4}(\beta^- + \beta^+)^2}
    <  \frac{1}{2}\sqrt{1+(\beta^-)^2}
    + \frac{1}{2} \sqrt{1+\left(\frac{U_\alpha(x_0+\mu) - U_\alpha(x_0)}{\mu}\right)^2}.
\end{equation}

The preceding estimate follows from \eqref{eq:strictconv} once we obtain that
\begin{equation}\label{eq:tousecxty}
	\frac{\mu-\lambda}{2\lambda}<s_0 \enspace \hbox{ and } \enspace \left| \frac{U_\alpha(x_0+\mu) - U_\alpha(x_0)}{\mu} - \beta^+  \right| < \eps_0.
\end{equation}
By \eqref{eq:cont4} and \eqref{eq:cont3d} (see also \eqref{eq:cont3c}), we have that
\begin{equation*}
	\frac{\mu-\lambda}{2\lambda} = \tfrac{1}{2}\big( \tfrac{\beta^+}{\beta_{\bar y}}-1\big) < s_0  \hbox{ and } \enspace \left| \frac{U_\alpha(x_0+\mu) - U_\alpha(x_0)}{\mu} - \beta^+  \right| =|\beta_{\bar y} - \beta^+| =\beta^+ - \beta_{\bar y}  < \eps_0.
\end{equation*}
 Hence, \eqref{eq:tousecxty} holds, and so does \eqref{eq:ineq_energy_competitor_4} by \eqref{eq:strictconv}. This concludes the proof of (iii), as well of the $\beta^+>0$ case.

Finally, we consider the case where $\beta^+\leq 0$. In this case, the competitor is defined for $x\in [0,1]$ by
\begin{equation*}
\begin{aligned}
\widetilde U_\alpha (x):=
\begin{cases}
U_\alpha(x) & \text{if } x < x_0 - \lambda,\\
U_\alpha(x_0 - \lambda) + \frac{1}{2} (\beta^+ + \beta^-)\big(x-(x_0-\lambda)\big)  & \text{if } x_0-\lambda \leq x < x_0 + \lambda, \\
\displaystyle\frac{U_\alpha(x_0+\mu)-\alpha - U_\alpha(x_0)-\beta\lambda}{\mu-\lambda}(x-x_0-\lambda)
    + U_\alpha(x_0)+\beta\lambda & \text{if } x_0 + \lambda \leq x \leq x_0 + \mu ,\\
U_\alpha(x) & \text{if } x > x_0 + \mu.
\end{cases}
\end{aligned}
\end{equation*}
In this case, though, we have the freedom to choose $\lambda$ and $\mu$ as we please. We therefore take $\lambda$ sufficiently small, and $\mu=\lambda+\lambda^2$.
Using similar arguments as in the previous case, we conclude.

\textbf{Step 2.}
We now consider the case in  (DR). The proof follows by applying the argument of Step 1 to the function $V_\alpha(x):= U_\alpha(2x_0-x)$.

\textbf{Step 3.}
We prove case in (UL). The proof follows by applying the argument of Step 1 to the function $V_\alpha(x):= -U_\alpha(x)$.

\textbf{Step 4.}
We prove case in (UR). The proof follows by applying the argument of Step 2 to the function $V_\alpha(x):= -U_\alpha(x)$.
\end{proof}


We now prove the monotonicity of the contact sets.

\begin{theorem}\label{thm:contact_set_inclusion}
Under the assumptions of Theorem~\ref{thm:optcond}, it holds for all $0<\alpha_1<\alpha_2<+\infty$ that
\begin{equation}\label{eq:monCalphas}
C_{\alpha_1}^+\cup C_{\alpha_1}^- \supset C_{\alpha_2}^+\cup C_{\alpha_2}^-.
\end{equation}
\end{theorem}

\begin{proof}
We claim that for every $\bar{\alpha}>0$, there exists $\beta(\bar{\alpha})\in (0, \bar\alpha]$ such that
\begin{equation}\label{eq:monot_contact_set}
C_\alpha^+\cup C_\alpha^- \supset 
C_{\bar{\alpha}}^+\cup C_{\bar{\alpha}}^- \quad \hbox{for all } \alpha\in(\bar{\alpha}-\beta(\bar{\alpha}), \bar{\alpha}].
\end{equation}

Once the preceding  claim is proved, we obtain \eqref{eq:monCalphas} by showing that \eqref{eq:monot_contact_set} implies that
\begin{equation}\label{eq:maxbeta}
	\bar\beta := \max \Big\{ \beta\in (0,\bar\alpha]\colon C_\alpha^+\cup C_\alpha^- \supset 
	C_{\bar{\alpha}}^+\cup C_{\bar{\alpha}}^-  \hbox{ for all } \alpha\in(\bar{\alpha}- \beta, \bar{\alpha}]\Big\}
\end{equation}
coincides with $\bar\alpha$. In fact,  assume by \ contradiction that \eqref{eq:monot_contact_set} holds  but $\bar\beta < \bar\alpha$. Then, we set $\tilde \alpha:= \bar\alpha - \bar\beta$, and use  \eqref{eq:monot_contact_set} to find $\beta(\tilde\alpha) \in (0, \tilde\alpha]$ such that $C_\alpha^+\cup C_\alpha^- \supset 
C_{\tilde{\alpha}}^+\cup C_{\tilde{\alpha}}^-$ for all $\alpha \in (\tilde\alpha- \beta(\tilde\alpha), \tilde\alpha] = (\bar\alpha - (\bar\beta + \beta(\tilde\alpha)), \bar\alpha-\bar\beta)]$. Thus, using \eqref{eq:maxbeta} for $\bar\beta$, we conclude that $\gamma:=\bar\beta + \beta(\tilde\alpha)$ belongs to the set on the right-hand side of \eqref{eq:maxbeta}, which contradicts the maximality of $\bar\beta$.

\begin{figure}[!ht]
\begin{tikzpicture}
\begin{axis}[height=0.4\textwidth,width=0.515\textwidth, clip=true,
    xmin=0,xmax=1,
    ymin=-0.1,ymax=0.25,
    ytick={-0.1,-0.05,...,0.25},
    yticklabel style={/pgf/number format/.cd, fixed},
    xtick={0,0.2064,0.2825,0.4565,0.5440,0.7190,0.7850,1.0000},
    xticklabels={$x_1^-\!(\alpha)$,$x_1^+\!(\alpha)\ \ \ $,$\ \ \ x_2^+\!(\alpha)$,$x_2^-\!(\alpha)\ \ $,$\ \ x_3^-\!(\alpha)$,$x_3^+\!(\alpha)\ \ \ \,$,$\ \ \ \,x_4^+\!(\alpha)$,$x_4^-\!(\alpha)$},
    every x tick/.style={color=black, thick},
    enlargelimits={abs=1cm},
    axis lines = left,
    axis line style={=>},
    xlabel style={at={(ticklabel* cs:1)},anchor=north west},
    ylabel style={at={(ticklabel* cs:1)},anchor=south west}
    ]
\addplot[samples=100, domain=0:1, black, thick] {-cos(4*pi*deg(x+0.25))/4/pi+sin(pi*deg(x+0.25)/2)/pi-0.202-0.03} node[anchor=west, pos=0.85, font=\footnotesize]{$F-\alpha$};
\addplot[samples=100, domain=0:1, black, thick] {-cos(4*pi*deg(x+0.25))/4/pi+sin(pi*deg(x+0.25)/2)/pi-0.202+0.03} node[anchor=south, pos=0.5, font=\footnotesize]{$F+\alpha$};
\addplot[samples=2, color=OliveGreen, very thick] coordinates {(0,0) (0.2064,-0.0308)};
\addplot[samples=20, domain=0.206:0.284, color=OliveGreen, very thick] {-cos(4*pi*deg(x+0.25))/4/pi+sin(pi*deg(x+0.25)/2)/pi-0.202+0.03};
\addplot[samples=2, color=OliveGreen, very thick] coordinates {(0.2825,-0.0088) (0.4565,0.121)};
\addplot[samples=20, domain=0.456:0.545, color=OliveGreen, very thick] {-cos(4*pi*deg(x+0.25))/4/pi+sin(pi*deg(x+0.25)/2)/pi-0.202-0.03}  node[anchor=south, pos=0.51, font=\footnotesize]{$U_\alpha$};
\addplot[samples=2, color=OliveGreen, very thick] coordinates {(0.544,0.1374) (0.719,0.0723)};
\addplot[samples=20, domain=0.717:0.79, color=OliveGreen, very thick] {-cos(4*pi*deg(x+0.25))/4/pi+sin(pi*deg(x+0.25)/2)/pi-0.202+0.03};
\addplot[samples=2, color=OliveGreen, very thick] coordinates {(0.785,0.0738) (1,0.1718)};
\draw[color=OliveGreen, thick, dotted] (axis cs:0.0020, -0.098) -- (axis cs:0.0020,  0.0000);
\draw[color=OliveGreen, thick, dotted] (axis cs:0.2064, -0.098) -- (axis cs:0.2064, -0.0308);
\draw[color=OliveGreen, thick, dotted] (axis cs:0.2825, -0.098) -- (axis cs:0.2825, -0.0088);
\draw[color=OliveGreen, thick, dotted] (axis cs:0.4565, -0.098) -- (axis cs:0.4565,  0.1210);
\draw[color=OliveGreen, thick, dotted] (axis cs:0.5440, -0.098) -- (axis cs:0.5440,  0.1374);
\draw[color=OliveGreen, thick, dotted] (axis cs:0.7190, -0.098) -- (axis cs:0.7190,  0.0723);
\draw[color=OliveGreen, thick, dotted] (axis cs:0.7850, -0.098) -- (axis cs:0.7850,  0.0738);
\draw[color=OliveGreen, thick, dotted] (axis cs:0.9980, -0.098) -- (axis cs:0.9980,  0.1718);
\draw[color=OliveGreen, ultra thick] (axis cs:0.2064, -0.098) -- (axis cs:0.2825,  -0.098);
\node[color=OliveGreen, font=\scriptsize] at (axis cs:0.248,-0.06) {$C_\alpha^+$};
\draw[color=OliveGreen, ultra thick] (axis cs:0.4565, -0.098) -- (axis cs:0.5440,  -0.098)
 node[anchor=south, pos=0.55, font=\scriptsize]{$C_\alpha^-$};
\draw[color=OliveGreen, ultra thick] (axis cs:0.7190, -0.098) -- (axis cs:0.7850,  -0.098)
 node[anchor=south, pos=0.55, font=\scriptsize]{$C_\alpha^+$};
\end{axis}
\end{tikzpicture}
\raisebox{0.18cm}{\begin{tikzpicture}
\begin{axis}[height=0.4\textwidth,width=0.515\textwidth, clip=true,
    xmin=0,xmax=1,
    ymin=-0.15,ymax=0.25,
    ytick={-0.15,-0.10,...,0.25},
    yticklabel style={/pgf/number format/.cd, fixed},
    xtick={0.109,0.3695,0.632,0.889},
    xticklabels={$m_1$,$m_2$,$m_3$,$m_4$},
    every x tick/.style={color=black, thick},
    enlargelimits={abs=1cm},
    axis lines = left,
    axis line style={=>},
    xlabel style={at={(ticklabel* cs:1)},anchor=north west},
    ylabel style={at={(ticklabel* cs:1)},anchor=south west}
    ]
\addplot[samples=100, domain=0:1, black, thick] {-cos(4*pi*deg(x+0.25))/4/pi+sin(pi*deg(x+0.25)/2)/pi-0.202-0.06} node[anchor=east, pos=0.72, font=\footnotesize]{$F\!-\!\bar\alpha$};
\addplot[samples=100, domain=0:1, black, thick] {-cos(4*pi*deg(x+0.25))/4/pi+sin(pi*deg(x+0.25)/2)/pi-0.202} node[anchor=south, pos=0.548, font=\footnotesize]{$F\!+\!2\alpha\!-\!\bar\alpha$};
\addplot[samples=100, domain=0:1, gray, thick] {-cos(4*pi*deg(x+0.25))/4/pi+sin(pi*deg(x+0.25)/2)/pi-0.202+0.06} node[anchor=east, pos=0.35, font=\footnotesize]{$F\!+\!\bar\alpha$};
\addplot[samples=2, color=OliveGreen, very thick] coordinates {(0,0-0.03) (0.2064,-0.0308-0.03)};
\addplot[samples=20, domain=0.206:0.284, color=OliveGreen, very thick] {-cos(4*pi*deg(x+0.25))/4/pi+sin(pi*deg(x+0.25)/2)/pi-0.202+0.0};
\addplot[samples=2, color=OliveGreen, very thick] coordinates {(0.2825,-0.0088-0.03) (0.4565,0.121-0.03)};
\addplot[samples=20, domain=0.456:0.545, color=OliveGreen, very thick] {-cos(4*pi*deg(x+0.25))/4/pi+sin(pi*deg(x+0.25)/2)/pi-0.202-0.06};
\addplot[samples=2, color=OliveGreen, very thick] coordinates {(0.544,0.1374-0.03) (0.719,0.0723-0.03)};
\addplot[samples=20, domain=0.717:0.79, color=OliveGreen, very thick] {-cos(4*pi*deg(x+0.25))/4/pi+sin(pi*deg(x+0.25)/2)/pi-0.202+0.0};
\addplot[samples=2, color=OliveGreen, very thick] coordinates {(0.785,0.0738-0.03) (1,0.1718-0.03)};
\node[color=OliveGreen, font=\scriptsize] at (axis cs:0.896,0.04) {$U_\alpha\!+\!\alpha\!-\!\bar\alpha$};
\draw[color=OliveGreen, thick, dotted] (axis cs:0.2064, -0.15) -- (axis cs:0.2064, -0.0308-0.03);
\draw[color=OliveGreen, thick, dotted] (axis cs:0.2825, -0.15) -- (axis cs:0.2825, -0.0088-0.03);
\draw[color=OliveGreen, thick, dotted] (axis cs:0.4565, -0.15) -- (axis cs:0.4565,  0.1210-0.03);
\draw[color=OliveGreen, thick, dotted] (axis cs:0.5440, -0.15) -- (axis cs:0.5440,  0.1374-0.03);
\draw[color=OliveGreen, thick, dotted] (axis cs:0.7190, -0.15) -- (axis cs:0.7190,  0.0723-0.03);
\draw[color=OliveGreen, thick, dotted] (axis cs:0.7850, -0.15) -- (axis cs:0.7850,  0.0738-0.03);
\draw[color=OliveGreen, ultra thick] (axis cs:0.2064, -0.148) -- (axis cs:0.2825,  -0.148);
\draw[color=OliveGreen, ultra thick] (axis cs:0.4565, -0.148) -- (axis cs:0.5440,  -0.148);
\draw[color=OliveGreen, ultra thick] (axis cs:0.7190, -0.148) -- (axis cs:0.7850,  -0.148);
\addplot[samples=2, color=RoyalBlue, very thick] coordinates {(0,0.03-0.03) (0.221,0.0282-0.03)};
\addplot[samples=20, domain=0.219:0.258, color=RoyalBlue, very thick] {-cos(4*pi*deg(x+0.25))/4/pi+sin(pi*deg(x+0.25)/2)/pi-0.202+0.06};
\addplot[samples=2, color=RoyalBlue, very thick] coordinates {(0.257,0.0362-0.03) (0.482,0.136-0.03)};
\addplot[samples=20, domain=0.481:0.520, color=RoyalBlue, very thick] {-cos(4*pi*deg(x+0.25))/4/pi+sin(pi*deg(x+0.25)/2)/pi-0.202-0.06}  node[anchor=south, pos=0.51, font=\footnotesize]{$U_{\bar\alpha}$};
\addplot[samples=2, color=RoyalBlue, very thick] coordinates {(0.519,0.143-0.03) (0.745,0.1267-0.03)};
\addplot[samples=20, domain=0.744:0.778, color=RoyalBlue, very thick] {-cos(4*pi*deg(x+0.25))/4/pi+sin(pi*deg(x+0.25)/2)/pi-0.202+0.06};
\addplot[samples=2, domain=0.778:1, color=RoyalBlue, very thick] coordinates {(0.777,0.131-0.03) (1,0.202-0.03)};
\draw[color=RoyalBlue, thick, dotted] (axis cs:0.2210, -0.147) -- (axis cs:0.2210, 0.0282-0.03);
\draw[color=RoyalBlue, thick, dotted] (axis cs:0.2570, -0.147) -- (axis cs:0.2570, 0.0362-0.03);
\draw[color=RoyalBlue, thick, dotted] (axis cs:0.4820, -0.147) -- (axis cs:0.4820, 0.1360-0.03);
\draw[color=RoyalBlue, thick, dotted] (axis cs:0.5190, -0.147) -- (axis cs:0.5190, 0.1430-0.03);
\draw[color=RoyalBlue, thick, dotted] (axis cs:0.7450, -0.147) -- (axis cs:0.7450, 0.1267-0.03);
\draw[color=RoyalBlue, thick, dotted] (axis cs:0.7770, -0.147) -- (axis cs:0.7770, 0.1310-0.03);
\draw[color=RoyalBlue, ultra thick] (axis cs:0.2210, -0.144) -- (axis cs:0.2570,  -0.144);
\node[color=RoyalBlue, font=\scriptsize] at (axis cs:0.32,-0.132) {$C_{\bar\alpha}^+$};
\draw[color=RoyalBlue, ultra thick] (axis cs:0.4820, -0.144) -- (axis cs:0.5190,  -0.144);
\node[color=RoyalBlue, font=\scriptsize] at (axis cs:0.585,-0.132) {$C_{\bar\alpha}^-$};
\draw[color=RoyalBlue, ultra thick] (axis cs:0.7450, -0.144) -- (axis cs:0.7750,  -0.144);
\node[color=RoyalBlue, font=\scriptsize] at (axis cs:0.825,-0.132) {$C_{\bar\alpha}^+$};
\end{axis}
\end{tikzpicture}}
\caption{The setup of the proof of Theorem \ref{thm:contact_set_inclusion} for $f$ as in Figure \ref{fig:rofvstaut}, $\alpha = 0.03$ and $\bar\alpha = 0.06$.}
\label{fig:contact_inclusion}
\end{figure}
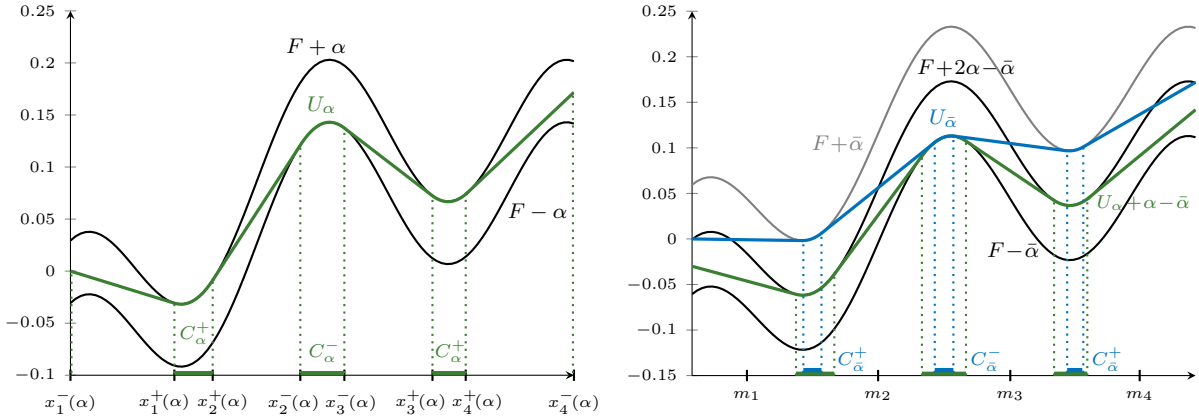

Consequently, we are left to prove the claim in  \eqref{eq:monot_contact_set}, illustrated in Figure \ref{fig:contact_inclusion}.
This will be done in three steps. Before entering into the details, we present the idea of the strategy.
First of all, we notice that it is possible to localize the problem, and looking only at intervals of the form $[m_i,m_{i+1}]$ where either $U_\alpha < F + \alpha$, or $U_\alpha > F - \alpha$. This is done in Step 1.
To fix the ideas, let us suppose that $U_\alpha < F + \alpha$ in the interval $[m_i,m_{i+1}]$.
In order to visualize better what happens, we assume that the lower obstacle is fixed when we move the parameter $\alpha$.
Then, in Step~2 we show that this localized problem solved by $U_\alpha$ is equivalent to taking the concave envelope of the function given by the lower obstacle $F-\alpha$ in $(m_i,m_{i+1})$ and the boundary conditions of $U_\alpha$ at the points $\{m_i,m_{i+1}\}$.
We now focus on the behavior of the boundary conditions, and  we prove in Step~3 that they are monotone with respect to $\alpha$.
This allows us to reach the desired conclusion in Step~4.

\textbf{Step 1.} For each $\alpha>0$,  part 3 of Lemma \ref{lem:geomtaut} yields that there are only finitely many, say $k(\alpha)\in\N$, pairs of points $x^-_i(\alpha)$,  $x^+_i(\alpha)\in[0,1]$ satisfying for all  $i\in\{1,\dots, k(\alpha)\}$ the following four properties:
\begin{itemize}
\item[(i)] $x^-_i(\alpha) \in C_{\bar\alpha}^-\cup\{0,1\}$, $x^+_i(\alpha) \in C_{\bar\alpha}^+\cup\{0,1\}$;
\item[(ii)] $U_{\bar\alpha}$ is affine on the interval having $x_i^-(\alpha)$ and $x_i^+(\alpha)$ as endpoints;
\item[(iii)] The open interval having $x_i^-(\alpha)$ and $x_i^+(\alpha)$ as endpoints does not intersect $C^-_\alpha\cup C^+_\alpha$;
\item[(iv)] $| x^-_i(\alpha) - x^+_i(\alpha) | \geq \frac{\mathrm{Lip}(F)}{\alpha}$ whenever $x_i^\pm(\alpha) \in C_\alpha^\pm$.
\end{itemize}
Note that we do not write the explicit interval with $x_i^-(\alpha)$ and $x_i^+(\alpha)$ as endpoints because each of them alternates between being a left endpoint and a right endpoint.
Moreover, we observe that not all points in the boundary of the contact set are used, but only those where the solution transitions between the lower and the upper obstacle (see Figure \ref{fig:choice_of_points}), and that the first and the last such points are $0$ and $1$, respectively.

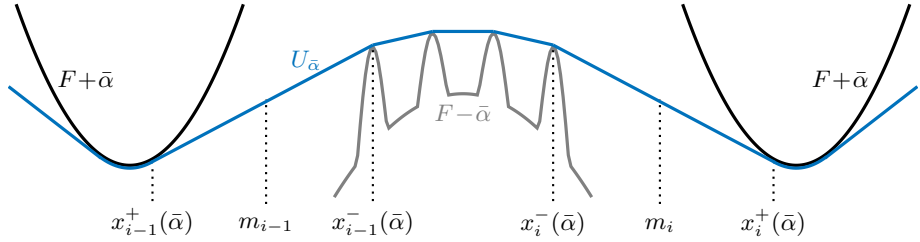
\begin{figure}[!ht]
\begin{tikzpicture}
\begin{axis}[height=0.3\textwidth,width=0.85\textwidth, clip=true,
    axis x line=none,
    axis y line=none,
    xmin=-3.0,xmax=3.0,
    ymin=-1.0,ymax=0.7
    ]
\addplot[samples=100, domain=-2.95:-1.45, black, very thick] {-0.5+2*abs((x+2.2)*(x+2.2))} node[anchor=west, pos=0.2, font=\small]{$F\!+\!{\bar\alpha}$};
\addplot[samples=100, domain=-0.85:0.85, gray, very thick] {-0.7*(cos(5*pi*deg(x)) <= 0.0)*(abs(x)<=0.7)*(0.5+0.8*abs(x))*cos(5*pi*deg(x))-abs(x*x)} node[anchor=north, pos=0.5, font=\small]{$F\!-\!{\bar\alpha}$};
\addplot[samples=100, domain=1.45:2.95, black, very thick] {-0.5+2*abs((x-2.2)*(x-2.2))} node[anchor=east, pos=0.8, font=\small]{$F\!+\!{\bar\alpha}$};
\addplot[samples=2, color=RoyalBlue, very thick] coordinates {(-3.0,0.05) (-2.39,-0.45)};
\addplot[samples=20, domain=-2.4:-2.05, color=RoyalBlue, very thick] {-0.52+2*abs((x+2.2)*(x+2.2))};
\addplot[samples=2, color=RoyalBlue, very thick] coordinates {(-2.06,-0.48) (-0.595,0.34)} node[anchor=south, pos=0.7, color=RoyalBlue, font=\small]{$U_{\bar\alpha}$};
\addplot[samples=2, color=RoyalBlue, very thick] coordinates {(-0.596,0.34) (-0.2,0.435)};
\addplot[samples=2, color=RoyalBlue, very thick] coordinates {(-0.201,0.435) (0.201,0.435)};
\addplot[samples=2, color=RoyalBlue, very thick] coordinates {(0.2,0.435) (0.596,0.34)};
\addplot[samples=2, color=RoyalBlue, very thick] coordinates {(0.595,0.34) (2.06,-0.48)};
\addplot[samples=20, domain=2.05:2.4, color=RoyalBlue, very thick] {-0.52+2*abs((x-2.2)*(x-2.2))};
\addplot[samples=2, color=RoyalBlue, very thick] coordinates {(2.39,-0.45) (3.0,+0.05)};
\draw[black, thick, dotted] (axis cs:-2.05, -0.75) -- (axis cs:-2.05,  -0.49);
\node[black, font=\small] at (axis cs:-2.05,-0.9){$x_{i-1}^+({\bar\alpha})$};
\draw[black, thick, dotted] (axis cs:-1.3, -0.77) -- (axis cs:-1.3,  -0.04);
\node[black, font=\small] at (axis cs:-1.3,-0.92){$m_{i-1}$};
\draw[black, thick, dotted] (axis cs:-0.595, -0.75) -- (axis cs:-0.595,  0.34);
\node[black, font=\small] at (axis cs:-0.595,-0.9){$x_{i-1}^-({\bar\alpha})$};
\draw[black, thick, dotted] (axis cs:0.595, -0.75) -- (axis cs:0.595,  0.34);
\node[black, font=\small] at (axis cs:0.595,-0.9){$x_{i}^-({\bar\alpha})$};
\draw[black, thick, dotted] (axis cs:1.3, -0.77) -- (axis cs:1.3,  -0.04);
\node[black, font=\small] at (axis cs:1.3,-0.92){$m_{i}$};
\draw[black, thick, dotted] (axis cs:2.05, -0.75) -- (axis cs:2.05,  -0.49);
\node[black, font=\small] at (axis cs:2.05,-0.9){$x_{i}^+({\bar\alpha})$};
\end{axis}
\end{tikzpicture}
\caption{Choice of points $x_i^\pm({\bar\alpha})$ and $m_i$ for the convex envelope argument of Theorem~\ref{thm:contact_set_inclusion}. In this example, not all points in $\partial C^-_{\bar\alpha}$ are used.}
\label{fig:choice_of_points}
\end{figure}

Fix $\bar{\alpha}>0$. For each $i=1,\dots, k(\bar\alpha)$, let
\[
m_i:= \frac{x^-_i(\bar{\alpha}) + x^+_i(\bar{\alpha})}{2},
\]
and fix
\[
0 < \delta < \min \left\{\frac{\big|x_i^-(\bar{\alpha}) - x_i^+(\bar{\alpha})\big|}{2} \,:\, i=1,\dots,k(\bar{\alpha}) \right\}.
\]
 Note that, by construction, each interval $[m_i, m_{i+1}]$ intersects $C^+_{\bar \alpha}$ or $C^-_{\bar \alpha}$, but not both. 
We claim that it is possible to choose $\beta(\bar\alpha) > 0$ such that for every $\alpha \in \big(\bar\alpha - \beta(\bar\alpha), \bar\alpha\big]$ the following hold:
\begin{itemize}
\item[(a)] $U_\alpha$ is affine in $(m_i - \delta, m_i + \delta)$;
\item[(b)] Each interval with endpoints $m_i$ and $m_{i+1}$ intersects $C^+_\alpha$ or $C^-_\alpha$, but not both;
\item[(c)] There exist $\widetilde{\delta}^{\pm}_i, \bar{\delta}^{\pm}_i>0$ such that
\begin{equation}\label{eq:position_contact_point_left}
x^-_i(\alpha)\in (x^-_i(\bar{\alpha}) - \widetilde{\delta}^-_i, x^-_i(\bar{\alpha}) + \bar{\delta}^-_i),
\end{equation}
\begin{equation}\label{eq:position_contact_point_right}
x^+_i(\alpha)\in (x^+_i(\bar{\alpha}) - \widetilde{\delta}^+_i, x^+_i(\bar{\alpha}) + \bar{\delta}^+_i),
\end{equation}
and
\begin{equation}\label{eq:position_contact_point_intersection}
(x^-_i(\bar{\alpha}) - \widetilde{\delta}^-_i, x^-_i(\bar{\alpha}) + \bar{\delta}^-_i) \cap (x^+_i(\bar{\alpha}) - \widetilde{\delta}^+_i, x^+_i(\bar{\alpha}) + \bar{\delta}^+_i) = \emptyset,
\end{equation}
for all $i\in\{1,\dots,k(\alpha)\}$;
\item[(d)] $k(\alpha)=k(\bar{\alpha})$.
\end{itemize}
We know that
\[
 F (x) - \bar{\alpha} < U_{\bar{\alpha}}(x) <  F (x) + \bar{\alpha},
\]
for every $x\in [m_i - \delta, m_i + \delta]$.
By Proposition \ref{prop:linfty_stability}, we have that
\begin{equation}\label{eq:unif}
U_\alpha\to U_{\bar{\alpha}} \quad\text{ uniformly as } \alpha\to\bar{\alpha}.
\end{equation}
Therefore, it is possible to find $\beta(\bar\alpha) > 0$ such that
\begin{equation}\label{eq:ineq_U_F}
 F (x) - \alpha < U_\alpha(x) <  F (x) + \alpha,
\end{equation}
for every $x\in (m_i - \delta, m_i + \delta)$, and every $\alpha \in \big(\bar\alpha - \beta(\bar\alpha), \bar\alpha\big]$.
By Lemma \ref{lem:geomtaut}, we know that $U_\alpha$ is affine in every connected component of $(0,1)\setminus (C^-_\alpha\cup C^+_\alpha)$.
Thus, by \eqref{eq:ineq_U_F}, we get that (a) is satisfied.

 To prove (b), we use the continuity of $F$ and $U_{\bar \alpha}$  to  fix $\eta>0$ such that: 
\begin{itemize}
\item[(I)] If $F(x)+\bar{\alpha} > U_{\bar{\alpha}}(x)$ for all $x\in [m_i,m_{i+1}]$, then $F(x)+\bar{\alpha}-\eta > U_{\bar{\alpha}}(x)$
for all $x\in [m_i,m_{i+1}]$;
\item[(II)] If $F(x)-\bar{\alpha} < U_{\bar{\alpha}}(x)$ for all $x\in [m_i,m_{i+1}]$, then $F(x) -\bar{\alpha}+\eta < U_{\bar{\alpha}}(x)$ for all $x\in [m_i,m_{i+1}]$.
\end{itemize}
 Note that either (I) or (II) hold because, as mentioned before, each interval $[m_i, m_{i+1}]$ intersects $C^+_{\bar \alpha}$ or $C^-_{\bar \alpha}$, but not both.  
Using \eqref{eq:unif}, and up to reducing $\beta(\bar{\alpha})$, we get that
\begin{equation}\label{eq:tub_ngbh}
U_{\bar{\alpha}}(x) - \eta < U_{\alpha}(x) < U_{\bar{\alpha}}(x) + \eta,\quad\quad\quad\text{ for all } x\in[0,1].
\end{equation}
Thanks to the choice of $\eta$, we obtain (b).

Now, we prove (c). We detail the proof of \eqref{eq:position_contact_point_left}, and observe that  \eqref{eq:position_contact_point_right} can be established similarly.
Consider $x^-_i(\bar{\alpha})$. Just to fix the ideas, assume that
\begin{equation}\label{eq:assumption_1}
(x^-_i(\bar{\alpha})-\delta, x^-_i(\bar{\alpha})) \subset (0,1)\setminus (C_{\bar{\alpha}}^-\cup C_{\bar{\alpha}}^+).
\end{equation}
The other case is treated similarly.
Up to further reducing the values of $\eta>0$  and $\beta(\bar{\alpha})$, the continuity of $F$ implies that
\[
F(x)-\bar{\alpha} < U_{\bar{\alpha}}(x)-\eta-\beta(\bar{\alpha}), \quad\quad\quad\text{ for all }x\in [m_i, x^-_i(\bar{\alpha}) - \delta].
\]
Therefore, we get from \eqref{eq:tub_ngbh}  that 
\[
F(x) - \alpha < F(x) - \bar{\alpha} + \beta(\bar{\alpha}) < U_{\bar{\alpha}}(x)-\eta < U_\alpha(x) \quad\quad\quad\text{ for all }x\in [m_i, x^-_i(\bar{\alpha}) - \delta].
\]
Thus, setting $\widetilde{\delta}^-_i:= \delta$,
\[
x^-_i(\alpha) > x^-_i(\bar{\alpha}) - \widetilde{\delta}^-_i.
\]
To prove that there exists $\bar{\delta}^-_i>0$ satisfying the required property, we consider two cases. First, we treat the situation where
\begin{equation}\label{eq:assumption_step1}
U'^{,-}_{\bar{\alpha}}(x^-_i(\bar{\alpha}))
    > U'^{,-}_{\bar{\alpha}}(x^-_{i+1}(\bar{\alpha})).
\end{equation}
Let
\[
y_i:= \max\Big\{ x\geq x^-_i(\bar{\alpha})\colon \,\, U'^{,-}_{\bar{\alpha}}(y) = U'^{,-}_{\bar{\alpha}}(x^-_i(\bar{\alpha})) \text{ for all } y\in [x^-_i(\bar{\alpha}), x] \Big\}
\]
and
\[
z_i:= \max\Big\{ x\leq x^-_{i+1}(\bar{\alpha})\colon\,\, U'^{,+}_{\bar{\alpha}}(y) = U'^{,+}_{\bar{\alpha}}(x^-_{i+1}(\bar{\alpha})) \text{ for all } y\in [x, x^-_{i+1}(\bar{\alpha})] \Big\}.
\]
Set
\[
\bar{\delta}^-_i := y_i-x^-_i(\bar{\alpha}) + \frac{z_i-y_i}{3}.
\]
Up to further reducing $\eta>0$, thanks to \eqref{eq:assumption_step1} and \eqref{eq:assumption_1}, we can also assume that
\begin{itemize}
\item[(III)] If $F(x)+\bar{\alpha} > U_{\bar{\alpha}}(x)$ for all $x\in [m_i,m_{i+1}]$, let $P_i\in[y_i,z_i]$ be such that
\[
U_{\bar{\alpha}}(y_i) + U'^{,-}_{\bar{\alpha}}(y_i)(P_i-y_i)
=
U_{\bar{\alpha}}(z_i) + U'^{,+}_{\bar{\alpha}}(z_i)(P_i-z_i).
\]
Then,
\[
F(P_i)+\eta < U_{\bar{\alpha}}(y_i) + U'^{,-}_{\bar{\alpha}}(y_i)(P_i-y_i);
\]
\item[(IV)] If $F(x)-\bar{\alpha}  <  U_{\bar{\alpha}}(x)$ for all $x\in [m_i,m_{i+1}]$, then let $Q_i\in[0,1]$ be such that
\[
U_{\bar{\alpha}}(y_i) + U'^{,-}_{\bar{\alpha}}(y_i)(Q_i-y_i)
=
U_{\bar{\alpha}}(z_i) + U'^{,+}_{\bar{\alpha}}(z_i)(Q_i-z_i).
\]
Then,
\[
F(Q_i)-\eta > U_{\bar{\alpha}}(z_i) + U'^{,-}_{\bar{\alpha}}(z_i)(Q_i-z_i).
\]
\end{itemize}
Therefore, (c) follows from \eqref{eq:tub_ngbh}, since $U_\alpha$ cannot be affine in the entire interval $[x^-_i(\bar{\alpha}), y_i]$.

We now consider the case where \eqref{eq:assumption_step1} does not hold.
By Lemma \ref{lem:geomtaut}, $U_{\bar{\alpha}}$ is convex in $[x^-_{i-1}(\bar{\alpha}), x^-_i(\bar{\alpha})]$. Therefore, the only other possibility is that
\[
U'^{,-}_{\bar{\alpha}}(x^-_i(\bar{\alpha}))
    = U'^{,-}_{\bar{\alpha}}(x^-_{i+1}(\bar{\alpha})).
\]
In this case, $U_{\bar{\alpha}}$ is affine in $[x^-_i(\bar{\alpha}), x^-_{i+1}(\bar{\alpha})]$.
By Proposition \ref{prop:linfty_stability}, Lemma \ref{lem:geomtaut} (DL), and the choice of $\eta$ in (I) and (II), only two possibilities can occur: either $x^-_i(\alpha)\in [x^-_i(\bar{\alpha}) - \delta, x^-_i(\bar{\alpha})]$ or $C_\alpha^-\cap [x_i^+(\alpha), x_{i+1}^+(\alpha)]=\emptyset$.
 We aim to exclude the second possibility. Assume, by contradiction, that it holds. Then, a symmetry argument allow us to assume \eqref{eq:ineq+1} without loss of generality.  Indeed, if the opposite inequality were true, we could proceed analogously using the function $x\mapsto U_\alpha(1-x)$.
At this stage, we choose to refer the reader to Case~2 of Step~3.1 for the final part of the argument, which is identical, instead of repeating it. Our choice is based on the fact that the strategy used in Step~3 forms a central part of the proof of   Theorem~\ref{thm:contact_set_inclusion}, and for the sake of clarity and readability, we prefer to maintain the logical flow of Step~3 without unnecessary interruptions.

Finally, (d) follows directly from  (c)  and the choice of $\delta$.

Fix $i\in\{1,\dots,k(\alpha)\}$.
Now, just to fix the ideas, assume that 
\begin{equation}\label{eq:contactbelow}
(m_{i-1},m_i)\cap C^+_\alpha =\emptyset.\end{equation}
Then, it is easy to see that $\restr{U_\alpha}{[m_{i-1}, m_i]}$ is the unique solution to
\begin{equation}\label{eq:localization}
\min \bigg\{\! \int_0^1 \sqrt{1 + |v'|^2} \dd x \colon v\in AC(m_{i-1},m_i),\, v \geq F - {\alpha},\  v(m_i)=U_\alpha(m_{i-1}), \ v(m_i)=U_\alpha(m_{i+1})\! \bigg\}.
\end{equation}
Notice that arguing as in \cite[Lemma~2]{Gra07}, this minimization problem has a unique solution.

\textbf{Step 2.} We claim that for each $\alpha \in \big(\bar\alpha - \beta(\bar\alpha), \bar\alpha\big]$, the unique solution to \eqref{eq:localization}, and hence $\restr{U_\alpha}{[m_{i-1},m_i]}$, is the concave envelope of the function
$G_\alpha:[m_{i-1}, m_i] \to \R$ defined as
\[
G_\alpha(x) = \begin{cases}U_\alpha(m_i) &\text{ if }x=m_{i-1},\\
F(x)-\alpha & \text{ if } x \in (m_{i-1}, m_i), \\
U_\alpha(m_{i+1}) &\text{ if }x=m_i.
\end{cases}
\]
For clarity, the proof of this fact will be done in two sub-steps.

\emph{Step 2.1.} First, we show the following. Let $v,w \in AC(m_{i-1},m_i)$ be such that $v(m_{i-1})=w(m_{i-1})$, $v(m_i)=w(m_i)$, and $v\leq w$.
Assume that $v$ is concave. Then, we claim that
\[
\int_{m_{i-1}}^{m_i} \sqrt{1+ |v'|^2} \dd t 
\leq \int_{m_{i-1}}^{m_i} \sqrt{1+ |w'|^2} \dd t.
\]
Under the additional assumption that $v\in W^{2,1}(0,1)\cap C^2(0,1)$,
consider the function $t\mapsto \sqrt{1+|t|^2}$. For $s,t\in\R$, we have that
\begin{align*}\sqrt{1+s^2} +  \frac{s}{\sqrt{1+s^2}} (t-s) &= \frac{1+st}{\sqrt{1+s^2}} = \frac{(1+st)\sqrt{1+t^2}}{\sqrt{(1+s^2)(1+t^2)}}\\ &= \frac{\sqrt{(1+2st+s^2t^2)(1+t^2)}}{\sqrt{1+s^2+t^2+s^2t^2}} \leq \sqrt{1+t^2},\end{align*}
where we have used $2st \leq s^2 + t^2$.
Thus, we get
\begin{align*}
\int_{m_{i-1}}^{m_i} \sqrt{1+ |w'|^2} \dd t
    &\geq \int_{m_{i-1}}^{m_i} \sqrt{1+ |v'|^2} \dd t
        + \int_{m_{i-1}}^{m_i} \frac{v'}{\sqrt{1+ |v'|^2}}(w'-v') \dd t \\
&= \int_{m_{i-1}}^{m_i} \sqrt{1+ |v'|^2} \dd t
    + \left.\frac{v'}{\sqrt{1+ |v'|^2}}(w-v)\right|_{m_{i-1}}^{m_{i}}\\
    &\qquad -\int_{m_{i-1}}^{m_i} \frac{v'' + (v''-1)(v')^2}{(1+|v'|^2)^{3/2}}(w-v) \dd t \\
&= \int_{m_{i-1}}^{m_i} \sqrt{1+ |v'|^2} \dd t
    -\int_{m_{i-1}}^{m_i} \frac{v'' + (v''-1)(v')^2}{(1+|v'|^2)^{3/2}}(w-v) \dd t \\
&\geq \int_{m_{i-1}}^{m_i} \sqrt{1+ |v'|^2} \dd t,
\end{align*}
where the second step follows from integration by parts, and the third by the assumption that $v$ and $w$ have the same boundary conditions.
Finally, the inequality is due to the concavity assumption of $v$, that ensures that $v''\leq 0$ when it exists, together with $w\geq v$.

To conclude in the general case, we use an approximation argument: let $\{v_n\}_n\in W^{2,1}(m_{i-1},m_i)\cap C^2(m_{i-1},m_i)$ and $\{w_n\}_n\in AC(m_{i-1},m_i)\cap C^2(m_{i-1},m_i)$ satisfying $v_n(m_i)=w_n(m_i)=v(m_i)$ and $v_n(m_{i-1})=w_n(m_{i-1})=v(m_{i-1})$ be such that $v_n\to v$ strongly in $W^{2,1}$ and $v_n\to v$ strongly in $W^{1,1}$. Then, by the previous argument, we get that 
\[
\int_{m_{i-1}}^{m_i} \sqrt{1+ |v'_n|^2} \dd t  \leq \int_{m_{i-1}}^{m_i} \sqrt{1+ |w'_n|^2} \dd t
\]
for all $n\in\N$.
Passing to the limit as $n\to\infty$ gives the desired result.

\emph{Step 2.2.} We now conclude as follows. First, we note that in the above argument, we could also require that $v\geq G_\alpha$. Indeed, the approximation step can be carried out without warring about this extra constraint.

Therefore, the conclusion of the previous step can be stated as follows: let $v, w\in AC(m_{i-1},m_i)$ with $v, w \geq F - {\alpha},\  v(m_i)=w(m_i)=U_\alpha(m_i), \ v(m_{i-1})=w(m_{i-1})=U_\alpha(m_{i-1})$.
Assume that $v$ is concave and that $v\leq w$. Then,
\[
\int_{m_{i-1}}^{m_i} \sqrt{1+ |v'|^2} \dd t  \leq \int_{m_{i-1}}^{m_i} \sqrt{1+ |w'|^2} \dd t.
\]
Thus, the solution of the constrained minimization problem \eqref{eq:localization} is given by the smallest concave function that is an admissible competitor. Namely, the concave envelope of $G_\alpha$.

An alternative proof of Step 2 under more restrictive assumptions can be found in \cite{KadTra01}.

\textbf{Step 3.} We now present the core of the proof; to be precise, a monotonicity property of the points $U_\alpha(m_i)$, when we consider  the graph of $F-\bar{\alpha}$ as a vertical reference.
We claim that
\begin{equation}\label{eq:monot_boundary_points}
U_{\alpha}(m_i) - (\bar{\alpha}-\alpha) \leq U_{\bar{\alpha}}(m_i),
\end{equation}
for all $\alpha \in (\bar\alpha - \beta(\bar\alpha), \bar\alpha]$. 

To prove \eqref{eq:monot_boundary_points}, it is convenient to define  the functions
\[
\widetilde{U}_\alpha:= U_\alpha - (\bar{\alpha}-\alpha),\quad\quad
\widetilde{F}_\alpha := F - (\bar{\alpha}-\alpha), \quad\quad \hbox{ for } \alpha<\bar{\alpha}.
\]
 Note that $\widetilde{U}_\alpha\leq U_\alpha$, $U_\alpha= F\pm \bar\alpha$ if and only if $\widetilde{U}_\alpha = \widetilde{F}_\alpha\pm\bar{\alpha}$, $\widetilde{U}_{\bar{\alpha}} = U_{\bar{\alpha}}$, and 
\begin{equation}\label{eq:conditions_U_tilde}
F-\bar{\alpha} = \widetilde{F}_\alpha-\alpha
\leq \widetilde{U}_\alpha
\leq \widetilde{F}_\alpha+\alpha = F+2\alpha-\bar{\alpha}.
\end{equation}

We will prove the validity of \eqref{eq:monot_boundary_points} in two sub-steps.

\textit{Step 3.1}: We claim that
\begin{equation}\label{eq:monot_contact_points}
x^-_i(\alpha)\geq x^-_i(\bar{\alpha}),\quad\quad\quad
x^+_i(\alpha)\leq x^+_i(\bar{\alpha}),
\end{equation}
for all $\alpha \in \big(\bar\alpha - \beta(\bar\alpha), \bar\alpha\big]$.

In case $x^-_i(\alpha)\in\{0,1\}$ or $x^+_i(\alpha)\in\{0,1\}$, this follows directly from the boundary conditions satisfied by each $U_\alpha$, for which we always have equality.
Next, we only detail the argument to prove that $x^-_i(\alpha)\geq x^-_i(\bar{\alpha})$ for all $\alpha \in \big(\bar\alpha - \beta(\bar\alpha), \bar\alpha\big]$. The proof of $x^+_i(\alpha)\leq x^+_i(\bar{\alpha})$ follows by arguing similarly.

Note that by \eqref{eq:contactbelow}, we have that
\[
x^-_i(\bar{\alpha}) < x^+_i(\bar{\alpha}).
\]
We start by making two observations. First, we claim that
\begin{equation}\label{eq:ineq_conc_left}
F(x) - \bar{\alpha} \leq U_{\bar{\alpha}}(x^-_i(\bar{\alpha}))
    + f^{-}(x^-_i(\bar{\alpha}))(x-x^-_i(\bar{\alpha})),
\end{equation}
for all $x\in(x^-_i(\bar{\alpha})-\delta, x^-_i(\bar{\alpha}))$, where $\delta$ is the one we fixed in Step 1.
Indeed, by the choice of $\delta$, we get that $(x^-_i(\bar{\alpha})-\delta, x^-_i(\bar{\alpha}))\subset (0,1) \setminus C^-_{\bar\alpha}$. Thus, thanks to (1) of Lemma \ref{lem:geomtaut}, we get that $U_{\bar\alpha}$ is concave in $x\in(x^-_i(\bar{\alpha})-\delta, x^-_i(\bar{\alpha}))$. Therefore,
\[
F(x)-\bar{\alpha} \leq U_{\bar{\alpha}}(x)
\leq U_{\bar{\alpha}}(x^-_i(\bar{\alpha})) + U'^{,-}_{\bar{\alpha}}(x^-_i(\bar{\alpha}))(x-x^-_i(\bar{\alpha}))
\leq U_{\bar{\alpha}}(x^-_i(\bar{\alpha}))
    + f^{-}(x^-_i(\bar{\alpha}))(x-x^-_i(\bar{\alpha})),
\]
where last inequality follows from (DR) of Theorem \ref{thm:optcond}.
By contradiction, assume that there exists $\alpha\in(\bar{\alpha}-\beta(\bar{\alpha}),\bar{\alpha}]$ such that $x^-_i(\alpha)<x^-_i(\bar{\alpha})$. Then, as we prove next, this implies  that
\begin{equation}\label{eq:ineq+1}
{U}'^{,+}_\alpha(x^-_i(\bar\alpha)) \geq
    U'^{,+}_{\bar{\alpha}}(x^-_i(\bar{\alpha})).
\end{equation}
Indeed, set
\begin{equation}\label{eq:ymin+}
	y_i^+:=\min \{ x^+_i(\alpha), x^+_i(\bar{\alpha}) \},
\end{equation}
 and observe that Lemma~\ref{lem:geomtaut} yields that $U_\alpha$ is affine in $[x_i^-(\alpha),y_i^+]$ and $U_{\bar\alpha}$ is affine in $[x_i^-(\bar\alpha),y_i^+]$.  Hence,
\begin{align}
	& \frac{U_\alpha(x_i^-(\bar\alpha))- U_\alpha(x_i^-(\alpha))}{x_i^-(\bar\alpha) - x_i^-(\alpha)} = U'^{,-}_{\alpha}(x_i^-(\bar\alpha)) = U'^{,+}_{\alpha}(x_i^-(\bar\alpha)) = \frac{U_\alpha(y^+_i)  - U_\alpha(x_i^-(\bar\alpha))}{ y^+_i - x_i^-(\alpha) } ,&\label{eq:aff1}\\
	&	U'^{,+}_{\bar\alpha}(x_i^-(\bar\alpha)) = \frac{U_{\bar\alpha}(y^+_i)  - U_{\bar\alpha}(x_i^-(\bar\alpha))}{ y^+_i - x_i^-(\alpha) }. &\label{eq:aff2}
\end{align}
Then,
\begin{flalign*}
&&   {U}'^{,+}_\alpha(x^-_i(\bar{\alpha})) = \widetilde{U}'^{,+}_\alpha(x^-_i(\bar{\alpha}))
    &=\frac{\widetilde{U}_\alpha(x^-_i(\bar{\alpha}))
        - \widetilde{U}_\alpha(x^-_i(\alpha))}{x^-_i(\bar{\alpha}) - x^-_i(\alpha)}
        && (\text{by \eqref{eq:aff1} and definition of $\widetilde{U}_\alpha$} ) \\
&& &\geq \frac{F(x^-_i(\bar{\alpha})) - \bar{\alpha}
    - \widetilde{U}_\alpha(x^-_i(\alpha))}{x^-_i(\bar{\alpha}) - x^-_i(\alpha)}
    && (\text{by } \eqref{eq:conditions_U_tilde}) \\
&& &=\frac{{U}_{\bar{\alpha}}(x^-_i(\bar{\alpha}))
    -  (F(x^-_i({\alpha})) - \bar{\alpha})}{x^-_i(\bar{\alpha}) - x^-_i(\alpha)}
    && (\text{definition of } x^-_i(\bar{\alpha}), x^-_i(\alpha) ) \\
&& &\geq f^{-}(x^-_i(\bar{\alpha}))
  && (\text{by  \eqref{eq:ineq_conc_left} at $x=x^-_i(\alpha)$}) \\
&& &\geq U'^{,+}_{\bar{\alpha}}(x^-_i(\bar{\alpha})) && (\text{by (DR) of Theorem ~\ref{thm:optcond}} ),
\end{flalign*}
which proves \eqref{eq:ineq+1}. 

\medskip
\textit{Case 1: $x_i^+(\bar\alpha) \leq x_i^+(\alpha) $.} Recalling \eqref{eq:ymin+},  we have in this case that $y^+_i=x_i^+(\bar\alpha)$. Then, using \eqref{eq:ineq+1},  \eqref{eq:aff1}, and \eqref{eq:aff2}, we deduce that $U_{\alpha}(x_i^+(\bar\alpha))  - U_{\alpha}(x_i^-(\bar\alpha)) \geq U_{\bar\alpha}(x_i^+(\bar\alpha))  - U_{\bar\alpha}(x_i^-(\bar\alpha))$ 
or, equivalently,
\begin{equation}\label{eq:ineq+2}
	U_{\alpha}(x_i^+(\bar\alpha))  - U_{\bar\alpha}(x_i^+(\bar\alpha)) \geq  U_{\alpha}(x_i^-(\bar\alpha))  - U_{\bar\alpha}(x_i^-(\bar\alpha)).
\end{equation}
To conclude, we use definition of  $x^\pm_i(\bar{\alpha})$ and $x^\pm_i(\alpha) $ and the admissibility conditions of $U_\alpha$ and $U_{\bar\alpha}$ to get that 
\begin{equation*}
	U_{\alpha}(x_i^-(\bar\alpha))  - U_{\bar\alpha}(x_i^-(\bar\alpha)) \geq F(x_i^-(\bar\alpha)) - \alpha -  (F(x_i^-(\bar\alpha)) - \bar\alpha ) =  \bar\alpha - \alpha,
\end{equation*}
while
\begin{equation*}
		U_{\alpha}(x_i^+(\bar\alpha))  - U_{\bar\alpha}(x_i^+(\bar\alpha)) \leq F(x^+_i(\bar \alpha))  + \alpha - (F(x^+_i(\bar\alpha))  + \bar\alpha)  = \alpha- \bar\alpha 
\end{equation*}
which contradicts \eqref{eq:ineq+2} because $\alpha<\bar\alpha$. 

\medskip
\textit{Case 2: $x_i^+(\bar\alpha)> x_i^+(\alpha) $.} In this case,  $y^+_i =x_i^+(\alpha)$ and $x_i^+(\alpha) \in ( x_i^-(\alpha), x_i^+(\bar\alpha) )$. Because the latter interval is contained in the set $(0,1)\setminus (C^-_{\bar \alpha} \cup C^+_{\bar \alpha})$, it follows that
\begin{equation}\label{eq:case2a}
	U_{\bar\alpha} (x_i^+(\alpha))  < F(x_i^+(\alpha))  + \bar\alpha.
\end{equation}
On the other hand, $U_{\alpha}$ is convex in the interval $[ x_i^+(\alpha), x_i^+(\bar\alpha) ]$ by Lemma~\ref{lem:geomtaut}; thus,
\begin{equation}\label{eq:case2b}
	\begin{aligned}
	 	U_{\alpha} (x_i^+(\bar \alpha)) &\geq 	U_{\alpha} (x_i^+(\alpha)) + {U}'^{,+}_\alpha(x^+_i(\alpha)) (x_i^+(\bar\alpha) - x_i^+(\alpha) )\\
	 	&\geq U_{\alpha} (x_i^+(\alpha)) + {U}'^{,+}_{\bar\alpha}(x^+_i(\alpha)) (x_i^+(\bar\alpha) - x_i^+(\alpha) ),
	\end{aligned}
\end{equation}
where, in the last estimate, we have used condition (UL) in Theorem ~\ref{thm:optcond} and the linearity of $U_\alpha$ in $[x^-_i(\bar\alpha), x^+_i(\alpha)]$ to say that 
\[{U}'^{,+}_\alpha(x^+_i(\alpha)) \geq {U}'^{,-}_\alpha(x^+_i(\alpha))={U}'^{,+}_\alpha(x^-_i(\bar\alpha)),\]
 while \eqref{eq:ineq+1}  and the linearity of $U_{\bar\alpha}$ in $[x^-_i(\bar\alpha), x^+_i(\bar\alpha)]$ yield 
 \[{U}'^{,+}_\alpha(x^-_i(\bar\alpha)) \geq  U'^{,+}_{\bar{\alpha}}(x^-_i(\bar{\alpha})) = {U}'^{,+}_{\bar\alpha}(x^+_i(\alpha)).\]
 Consequently,  using the definition of  $x^+_i(\bar{\alpha})$ and $x^+_i(\alpha) $ and the admissibility conditions of $U_\alpha$ and $U_{\bar\alpha}$, we conclude from \eqref{eq:case2a} and \eqref{eq:case2b} that
 \begin{equation*}
 		\begin{aligned}
 	 		F(x_i^+(\bar\alpha))  + \bar\alpha & \geq U_{\alpha} (x_i^+(\bar \alpha)) -\alpha+ \bar\alpha\\
 	 		& \geq U_{\alpha} (x_i^+(\alpha)) + {U}'^{,+}_{\bar\alpha}(x^+_i(\alpha)) (x_i^+(\bar\alpha) - x_i^+(\alpha) ) -\alpha+ \bar\alpha  \\
 	 		& = F(x_i^+(\alpha))  + \bar\alpha + {U}'^{,+}_{\bar\alpha}(x^+_i(\alpha)) (x_i^+(\bar\alpha) - x_i^+(\alpha) ) \\
 	 		& > U_{\bar\alpha} (x_i^+(\alpha)) + {U}'^{,+}_{\bar\alpha}(x^+_i(\alpha)) (x_i^+(\bar\alpha) - x_i^+(\alpha) ) \\
 	 		&=  U_{\bar\alpha} (x_i^+(\bar\alpha)) = F(x_i^+(\bar\alpha))  + \bar\alpha,
 	   \end{aligned}
 \end{equation*}
where the second to last equality follows from the linearity of $U_{\bar\alpha}$ in $[x^-_i(\bar\alpha), x^+_i(\bar\alpha)]$, which yields a contradiction.

This concludes the proof of \eqref{eq:monot_contact_points}.

\textit{Step 3.2.} We now prove that
\[
\widetilde{U}_\alpha \leq U_{\bar{\alpha}}\quad\text{ in } [x^-_i(\alpha), x^+_i(\alpha)].
\]
By Step 3.1, both functions are affine in $[x^-_i(\alpha), x^+_i(\alpha)]$. Thus, it suffices  show that 
\begin{equation}\label{eq:ineq_values_1}
	\widetilde{U}_\alpha(x^-_i(\alpha))\leq U_{\bar{\alpha}}(x^-_i(\alpha))
\end{equation}
and
\begin{equation}\label{eq:ineq_values_2}
	\widetilde{U}_\alpha(x^+_i(\alpha))\leq U_{\bar{\alpha}}(x^+_i(\alpha)).
\end{equation}
The estimate \eqref{eq:ineq_values_1} follows from the fact that
\[
\widetilde{U}_\alpha(x^-_i(\alpha))
= \widetilde{F}(x^-_i(\alpha))-\alpha
= F(x^-_i(\alpha))-\bar{\alpha}
\leq U_{\bar{\alpha}}(x^-_i(\alpha)).
\]
To prove \eqref{eq:ineq_values_2}, we argue as follows. In view of \eqref{eq:monot_contact_points}, we can reverse the roles of $U_\alpha$ and $U_{\bar\alpha}$ in \eqref{eq:ineq+1} to get 
\begin{equation*}
	{U}'^{,+}_{\bar{\alpha}}(x^-_i(\alpha)) \geq
	U'^{,+}_{{\alpha}}(x^-_i(\bar{\alpha})) = \widetilde{U}'^{,+}_{{\alpha}}(x^-_i(\bar{\alpha})),
\end{equation*}
which, by linearity, implies that $U_{\bar{\alpha}} (x_i^+(\alpha)) - U_{\bar{\alpha}} (x_i^-(\alpha)) \geq \widetilde{U}_{{\alpha}} (x_i^+(\alpha)) - \widetilde{U}_{{\alpha}} (x_i^-(\alpha)) $. Thus,
\begin{equation*}
U_{\bar{\alpha}} (x_i^+(\alpha)) - \widetilde{U}_{{\alpha}} (x_i^+(\alpha))  \geq  U_{\bar{\alpha}} (x_i^-(\alpha))  - \widetilde{U}_{{\alpha}} (x_i^-(\alpha)) \geq 0,
\end{equation*}
where we used \eqref{eq:ineq_values_1} in the last estimate. Hence, \eqref{eq:ineq_values_2} also holds.

If instead of \eqref{eq:contactbelow}, $i$ is such that
\begin{equation}\label{eq:contactabove}
(m_{i-1},m_i)\cap C_\alpha^- = \emptyset,
\end{equation}
by using the same symmetry argument as in Theorem \ref{thm:optcond}, we conclude the validity of \eqref{eq:monot_boundary_points} also in this case.
In particular, we obtain that
\begin{equation}\label{eq:ineq_U_alpha}
U_\alpha(m_i) - (\bar{\alpha}-\alpha) \leq U_{\bar{\alpha}}(m_i),
\end{equation}
for all $i\in\{1,\dots,k(\bar{\alpha})\}$, and for all $\alpha\in[\bar{\alpha}-\beta(\bar{\alpha}), \bar{\alpha}]$.

\textbf{Step 4.}
We now conclude as follows.
To fix the ideas, let's assume that \eqref{eq:contactbelow} is in force.
For each $\alpha\in(\bar{\alpha}-\beta(\bar{\alpha}), \bar{\alpha}]$,  Step~2 yields that $\restr{\widetilde{U}_\alpha}{[m_{i-1}, m_i]}$ is the concave envelope of  the function $\widetilde{G}_\alpha:[m_{i-1},m_i]\to\R$ defined by
\begin{equation}\label{eq:Gtilde}
	\widetilde{G}_{{\alpha}}(x) :=G_\alpha - (\bar{\alpha}-\alpha) =
	\begin{cases}
		\widetilde{U}_\alpha (m_{i-1}) & \hbox{if } x=m_{i-1},\\
		F(x) - \bar{\alpha} & \hbox{if } x\in (m_{i-1}, m_i),\\
		\widetilde{U}_\alpha (m_i) & \hbox{if } x=m_i.\\
	\end{cases} 
\end{equation}

On the other hand, by \eqref{eq:ineq_U_alpha}, we have for all $\alpha\in[\bar{\alpha}-\beta(\bar{\alpha}), \bar{\alpha}]$  that
\begin{equation}\label{eq:ineq_bp}
	\widetilde{U}_\alpha(m_{i-1}) \leq U_{\bar{\alpha}}(m_{i-1}),\quad\quad\quad
	\widetilde{U}_\alpha(m_i) \leq U_{\bar{\alpha}}(m_i).
\end{equation}
Thus,
\begin{equation}\label{eq:ineq_G}
	\widetilde{G}_\alpha \leq {G}_{\bar{\alpha}},
\end{equation}
which implies that $\widetilde{U}_\alpha \leq U_{\bar{\alpha}} $ by taking the concave envelope in the preceding estimate. Consequently,
\begin{equation}\label{eq:UUG}
	\widetilde{G}_\alpha \leq \widetilde{U}_\alpha \leq U_{\bar{\alpha}}.
\end{equation}
Hence, if $x\in C^-_{\bar{\alpha}}$, we get from \eqref{eq:Gtilde} that
\[
U_{\bar{\alpha}}(x) = F(x)-\bar{\alpha}= G_{\bar{\alpha}}(x)  = \widetilde{G}_\alpha(x).
\]
Therefore, using \eqref{eq:UUG}, we obtain that $\widetilde{U}_\alpha(x) =  F(x)-\bar{\alpha}$. Using the definition of $\widetilde{U}_\alpha$, this is equivalent to saying that $x\in C^-_{\alpha}$. Hence, $C^-_{\alpha}\supset C^-_{\bar{\alpha}}$.
If instead of \eqref{eq:contactbelow}, we have \eqref{eq:contactabove}, then  by using a symmetry argument, we get that $C^+_{\alpha}\supset C^+_{\bar{\alpha}}$.
\end{proof} 

We now investigate the relation between the jump set of $u_\alpha$ and that of $f$.

\begin{figure}[!ht]
\begin{tikzpicture}
\begin{axis}[height=0.35\textwidth,width=0.5\textwidth, clip=true,
    axis x line=none,
    axis y line=none,
    xmin=0.25,xmax=0.75,
    ymin=-0.025,ymax=0.5
    ]
\addplot[samples=20, domain=0.35:0.5, black, very thick] {4*x*x*x*x+0.2*(x-0.5)} ;
\addplot[samples=20, domain=0.5:0.75, black, very thick] {0.25-(x-0.55)*(x-0.55)+0.5*(x-0.5)} 
node[anchor=north, pos=0.6, font=\small]{$F\!-\!\alpha$};
\addplot[samples=2, color=OliveGreen, very thick] coordinates {(0.3,0.025) (0.505,0.2575)};
\addplot[samples=2, color=OliveGreen, densely dashed, semithick] coordinates {(0.505,0.2575) (0.605,0.370)}
node[anchor=west, pos=1, font=\small]{$U'^{,-}_\alpha(x)$};
\addplot[samples=2, black, densely dashed, semithick] coordinates {(0.505,0.2575) (0.605,0.4775)}
node[anchor=west, pos=1, font=\small]{$F'^{,-}(x)$};
\addplot[samples=2, color=OliveGreen, densely dashed, semithick] coordinates {(0.405,0.1975) (0.505,0.2575)}
node[anchor=east, pos=0, font=\small]{$U'^{,+}_\alpha(x)$};
\addplot[samples=20, domain=0.5:0.75, color=OliveGreen, very thick] {0.2575-(x-0.55)*(x-0.55)+0.5*(x-0.5)};
\node[color=OliveGreen, font=\small] at (axis cs:0.285,0.08){$U_\alpha$};
\draw[black, thick, dotted] (axis cs:0.5, 0.025) -- (axis cs:0.5,  0.25);
\node[black, font=\small] at (axis cs:0.575,0.0){$x \in J_{u_\alpha} \cap \partial C_\alpha^-$};
\end{axis}
\end{tikzpicture}
\caption{Jumps of $u_\alpha = U_\alpha'$ must be jumps of $f = F'$, even if they happen at the boundary of the contact set, cf. \eqref{eq:jumpinclusion_2} in Proposition \ref{prop:jump-inclusion}.}
\label{eq:jump_f_u}
\end{figure}
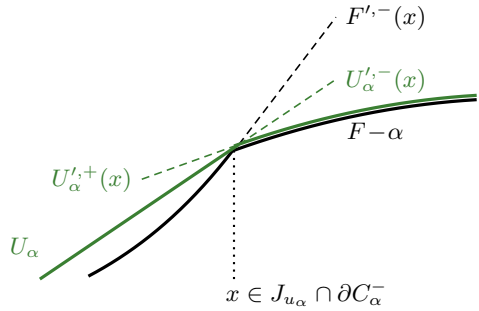

\begin{proposition}\label{prop:jump-inclusion}
Let $f \in L^\infty(0,1)$ be such that the approximate left and right limits in \eqref{eq:F'lim} exist for every $x_0 \in (0,1)$. 
Then, for every $\alpha\in(0,+\infty)$, we have that
\begin{equation}\label{eq:jumpinclusion}
J_{u_\alpha} \subset J_{f} \cap (C_{\alpha}^+\cup C_{\alpha}^-),
\end{equation}
where $J_f$ is (by definition) the set of points where both limits in \eqref{eq:F'lim} exist but are different.
More precisely, we have
\begin{equation}\label{eq:jumpinclusion_1}
J_{u_\alpha} \cap \interior (C_\alpha^+ \cup C_\alpha^-) =  J_{f} \cap \interior
(C_\alpha^+ \cup C_\alpha^-),
\end{equation}
as well as (see Figure \ref{eq:jump_f_u})
\begin{equation}\label{eq:jumpinclusion_2}
J_{u_\alpha} \cap \partial (C_\alpha^+ \cup C_\alpha^-) \subset J_f.
\end{equation}
\end{proposition}

\begin{proof}
We first prove \eqref{eq:jumpinclusion_1}.
We observe that, since \(x_0\in \interior C_\alpha^\pm \), it is possible to find \(\delta_{x_0}>0\) such that \(I_{x_0}:= (x_0-\delta_{x_0}, x_0 + \delta_{x_0}) \, \subset \interior C_\alpha^\pm\); in other words, using the  definition of \(C_\alpha^\pm\), such that
\begin{equation*}
\begin{aligned}
U_\alpha(x) = F(x) \pm \alpha \quad \text{ for all \(x\in I_{x_0}\)}.
\end{aligned}
\end{equation*}
Thus, recalling Remark~\ref{rmk:limU'} and using the definition of
\(F\) in \eqref{eq:defF},
we conclude that%
\begin{equation*}
\begin{aligned}
U_\alpha' = f \quad \Lcal^1\text{-a.e.~in }\, I_{x_0}.
\end{aligned}
\end{equation*}
Consequently, because \(x_0\in \interior C_\alpha^\pm \) was taken arbitrarily, we conclude that  \eqref{eq:jumpinclusion_1} holds,  with \(f\in BV\big(\interior\big(C_\alpha^+ \cup C_\alpha^-\big)\big)\)
by Theorem~\ref{thm:gra07}.

We now prove \eqref{eq:jumpinclusion_2}. Let us assume for a contradiction that there is some $x_0 \in (0,1)$ with
\[
x_0 \in \left( J_{u_\alpha} \cap \partial (C_\alpha^+ \cup C_\alpha^-) \right) \setminus J_f.
\]
Without loss of generality, we can assume that $x_0\in C^-_\alpha$, and that there exists $\delta > 0$ such that $U(x) > F(x)-\alpha$ for all $x \in (x_0-\delta,x_0)$. The proof in the other cases uses a similar strategy.
Thanks to case (DL) of Theorem \ref{thm:optcond}, we have that
\begin{equation}\label{eq:ineq_Uder_alpha}
U'^{,+}_\alpha(x_0) < U'^{,-}_\alpha(x_0),
\end{equation}
where the strict inequality is due to the assumption that $x_0\in J_{u\alpha}$.
Now, since $U\geq F-\alpha$, and $F(x_0)-\alpha=U_\alpha(x_0)$, we get that
\[
\frac{F(x)-F(x_0)}{x-x_0} \geq \frac{U_\alpha(x)-U_\alpha(x_0)}{x-x_0},
\]
for $x<x_0$, and
\[
\frac{F(x)-F(x_0)}{x-x_0} \leq \frac{U_\alpha(x)-U_\alpha(x_0)}{x-x_0}
\]
for $x>x_0$. Passing to the limit as $x\to x_0^-$ and $x\to x_0^+$, respectively, and using the absurd assumption that $x\notin J_f$, we get
\[
U'^{,+}_\alpha(x_0) \geq f(x_0) \geq U'^{,-}_\alpha(x_0),
\]
which is the desired contradiction with \eqref{eq:ineq_Uder_alpha}.

Finally, we show the validity of \eqref{eq:jumpinclusion}. By Theorem~\ref{thm:gra07}, we have that
\begin{equation*}
\begin{aligned}
J_{u_\alpha} = J_{U_\alpha'} 
.\end{aligned}
\end{equation*}
By Lemma~\ref{lem:geomtaut}, \(U'_\alpha\) is  constant on each connected component of the open set $(0,1)\setminus(C^-_\alpha \cup C^+_\alpha)$. Thus,
\begin{equation*}
\begin{aligned}
J_{U_\alpha'} \subset (0,1) \setminus \big((0,1)\setminus(C^-_\alpha \cup C^+_\alpha)\big) =C_\alpha^+ \cup C_\alpha^-.
\end{aligned}
\end{equation*}
Using \eqref{eq:jumpinclusion_1} and \eqref{eq:jumpinclusion_2}, we conclude the proof.
\end{proof}

\begin{remark}
In the case of ROF on higher dimensional domains, the inclusion \eqref{eq:jumpinclusion} can be proved by several different methods (either comparison of level sets seen as generalized surfaces of prescribed mean curvature as in \cite{CaChNo07} or different flavors of inner variations \cite{Val15, ChaLas24}), but in those cases one needs to assume $f \in L^\infty \cap BV$.
\end{remark}

\begin{remark}
We emphasize that the opposite inclusion to \eqref{eq:jumpinclusion_2} does not hold, as can be seen by taking $f = \1_{(1/2,1)}$ and $\alpha > 1/2$, in which case $u \equiv 1/2$.
\end{remark}

We finally conclude the proof by proving that the monotonicity of the contact set implies the monotonicity of the jump set.

\begin{proposition}\label{prop:from-contact-to-jump}
Let $\alpha_1,\alpha_2\in(0,+\infty)$ with $\alpha_1 < \alpha_2$, and assume that
\[
C_{\alpha_1}^+\cup C_{\alpha_1}^- \supset C_{\alpha_2}^+\cup C_{\alpha_2}^-.
\]
Then, \[J_{u_{\alpha_1}}\supset J_{u_{\alpha_2}}.\]
\end{proposition}

\begin{proof}
Let $x\in J_{u_{\alpha_2}}$. By \eqref{eq:jumpinclusion} of Proposition \ref{prop:jump-inclusion}, we get that $x\in J_f\cap (C_{\alpha_2}^+\cup C_{\alpha_2}^-)$. Thus, by our assumption, we infer that $x\in J_f\cap (C_{\alpha_1}^+\cup C_{\alpha_1}^-)$.
First, consider the case where $x\in J_f\cap \interior (C_{\alpha_1}^+\cup C_{\alpha_1}^-)$. By using \eqref{eq:jumpinclusion_1}, we obtain that $x\in J_{u_{\alpha_1}}$.

Now, consider the case where $x\in J_f\cap \partial (C_{\alpha_1}^+\cup C_{\alpha_1}^-)$. Assume $x\in J_f\cap \partial C_{\alpha_1}^-$. In the case $x\in J_f\cap \partial C_{\alpha_1}^+$, the conclusion follows by using a similar argument.
We argue by contradiction, namely we assume that $x\notin J_{u_{\alpha_1}}$. This yields that
\[
U'^{,+}_{{\alpha_1}}(x)=U'^{,-}_{{\alpha_1}}(x).
\]
We claim that one the following options holds:
\begin{equation}\label{eq:option1}
U'^{,+}_{\alpha_1}(x) > U'^{,+}_{\alpha_2}(x),
\end{equation}
or
\begin{equation}\label{eq:option2}
U'^{,-}_{\alpha_1}(x) > U'^{,-}_{\alpha_2}(x).
\end{equation}
Indeed, note that we must have
\begin{equation}\label{eq:optimality_alpha_2}
U'^{,+}_{\alpha_2}(x) < U'^{,-}_{\alpha_2}(x),
\end{equation}
where the strict inequality is due to the fact that $x\in J_{u_{\alpha_2}}$.
Thus, if \eqref{eq:option2} does not hold, than  $U'^{,-}_{\alpha_2}(x)\leq U'^{,-}_{\alpha_1}(x)$. Therefore, from \eqref{eq:optimality_alpha_2} we conclude that \eqref{eq:option1} is in force.

Thus, assume we are in the case where \eqref{eq:option1} holds. The case where \eqref{eq:option2} holds is treated with similar arguments.
We claim that there exists $x_{\alpha_1}>x$ such that $U_{\alpha_1}$ is affine in $[x,x_{\alpha_1}]$. Indeed, since
\[
U'^{,+}_{\alpha_1}(x) > U'^{,+}_{\alpha_2}(x), \quad\quad
U_{\alpha_1}(x)=F(x)-\alpha_1,\quad\quad
U_{\alpha_2}(x)=F(x)-\alpha_2,
\]
we obtain that there exists $x_0>x$ such that
\begin{equation}\label{eq:U_alpha1_affine}
U_{\alpha_1}(y) - (\alpha_2-\alpha_1) > U_{\alpha_2}(y),
\end{equation}
for all $y\in(x,x_0)$. Since $\alpha_1>0$, there exists $\delta>0$ such that $U_{\alpha_1}(y)\notin C^+_{\alpha_1}$ for $y\in[x,x+\delta]$. This proves the claim.

We can thus assume that $x_{\alpha_1}>x$ is the largest value for which $U_{\alpha_1}$ is affine in $[x,x_{\alpha_1}]$.
We then have three possible cases.

\begin{figure}[!ht]
\begin{tikzpicture}
\begin{axis}[height=0.4\textwidth,width=0.65\textwidth, clip=true,
    axis x line=none,
    axis y line=none,
    xmin=0.365,xmax=0.75,
    ymin=0.115,ymax=0.315
    ]
\addplot[samples=20, domain=0.485:0.5, black, very thick] {4*x*x*x*x+0.5*(x-0.5)};
\addplot[samples=20, domain=0.5:0.7, black, very thick] {4*(1-x)*(1-x)*(1-x)*(1-x)-0.9*((1-x)-0.5)} 
node[anchor=west, pos=1, font=\footnotesize]{$F\!-\!{\alpha_1}$};
\addplot[samples=20, domain=0.485:0.5, gray, very thick] {-0.025+4*x*x*x*x+0.5*(x-0.5)} ;
\addplot[samples=20, domain=0.5:0.7, gray, very thick] {-0.025+4*(1-x)*(1-x)*(1-x)*(1-x)-0.9*((1-x)-0.5)}
node[anchor=west, pos=1, font=\footnotesize]{$F\!-\!{\alpha_2}$};
\addplot[samples=20, domain=0.485:0.5002, black, very thick] {0.06+4*x*x*x*x+0.5*(x-0.5)};
\addplot[samples=20, domain=0.5:0.7, black, very thick] {0.06+4*(1-x)*(1-x)*(1-x)*(1-x)-0.9*((1-x)-0.5)}
node[anchor=west, pos=1, font=\footnotesize]{$F\!+\!{\alpha_1}$};
\addplot[samples=2, color=OliveGreen, very thick] coordinates {(0.44,0.2524) (0.7,0.2424)}
node[anchor=south, pos=0, font=\footnotesize]{$U_{\alpha_1}$};
\addplot[samples=2, color=RoyalBlue, very thick] coordinates {(0.44,0.21) (0.50,0.225)}
node[anchor=north, pos=0.0, font=\footnotesize]{$U_{\alpha_2}$};
\addplot[samples=50, domain=0.5:0.7, color=RoyalBlue, very thick] {0.25-0.04*sin(4*pi*deg(25*(-x+0.7)*(-x+0.7)*(-x+0.7)))-0.04*(x-0.5)};
\addplot[samples=2, color=Maroon, very thick] coordinates {(0.4,0.229) (0.7,0.2174)}
node[anchor=south, pos=0.13, font=\footnotesize]{$U_{\alpha_1}\!-\!(\alpha_2\!-\!\alpha_1)$};
\draw[black, thick, dotted] (axis cs:0.5, 0.135) -- (axis cs:0.5, 0.25);
\node[black, font=\small] at (axis cs:0.462,0.125){$J_{u_{\alpha_2}}\!\!\setminus\! J_{u_{\alpha_1}} \!\ni\! x $};
\draw[black, thick, dotted] (axis cs:0.572, 0.135) -- (axis cs:0.572, 0.225);
\node[black, font=\small] at (axis cs:0.574,0.125){$x_0$};
\draw[black, thick, dotted] (axis cs:0.7, 0.135) -- (axis cs:0.7, 0.2424);
\node[black, font=\small] at (axis cs:0.716,0.125){$x_{\alpha_1}\!=\!1$};
\end{axis}
\end{tikzpicture}
\caption{Case 1 in the proof of Proposition \ref{prop:from-contact-to-jump}, a competitor is built by replacing $U_{\alpha_2}$ by $U_{\alpha_1}-(\alpha_2-\alpha_1)$ on $[x,x_0]$.}
\label{fig:straightcompetitor}
\end{figure}
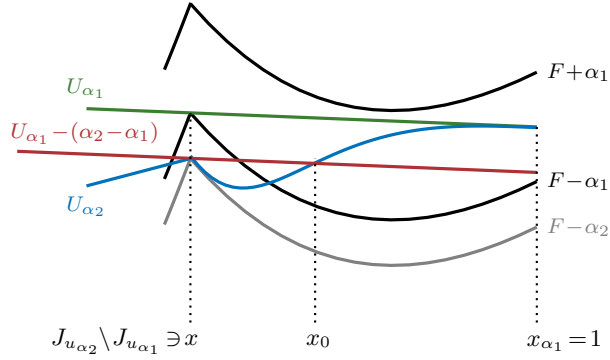

\emph{Case 1: $x_{\alpha_1}=1$}. Using \eqref{eq:U_alpha1_affine}, together with the fact that $U_{\alpha_1}(1)=F(1)=U_{\alpha_2}(1)$, we obtain that there exists $x_0\in(x,1]$ such that
\[
U_{\alpha_1}(x_0)-(\alpha_2-\alpha_1)=U_{\alpha_2}(x_0),
\]
and
\[
F(y)-\alpha_2\leq U_{\alpha_2}(y) < U_{\alpha_1}(y)-(\alpha_2-\alpha_1),
\]
for all $y\in[x,x_0)$. Note that, by \eqref{eq:option1} and the above, $U_{\alpha_2}$ cannot be affine in the entire interval $[x,x_0]$. But this then contradicts the minimality of $U_{\alpha_2}$, since in that case the graph of the function
\[
v(y):=
\left\{
\begin{array}{ll}
U_{\alpha_1}(y)-(\alpha_2-\alpha_1) & \text{ if } y\in[x,x_0],\\
U_{\alpha_2}(y) & \text{ else},
\end{array}
\right.
\]
has strictly less length than that of $U_{\alpha_2}$, as illustrated in Figure \ref{fig:straightcompetitor}. Note that the function $v$ is such that $v(0)=F(0)$, $v(1)=F(1)$, and $\|v-F\|_\infty\leq \alpha_2$. This gives the desired contradiction.

\emph{Case 2: $x_{\alpha_1}\in C^-_{\alpha_1}$}. We claim that there exists $x_0\in(x,1]$ such that
\[
U_{\alpha_1}(x_0)-(\alpha_2-\alpha_1) = U_{\alpha_2}(x_0),
\]
and
\[
F(y)-\alpha_2\leq U_{\alpha_2}(y) < U_{\alpha_1}(y)-(\alpha_2-\alpha_1),
\]
for all $y\in[x,x_0)$. This follows from the continuity of $F$, $U_{\alpha_1}$, and $U_{\alpha_2}$ together with the fact that $U_{\alpha_1}(x_{\alpha_1}))-(\alpha_2-\alpha_1)=F(x_{\alpha_1})-\alpha_2$. We then conclude with a similar argument to that employed in Case 1.

\emph{Case 3: $x_{\alpha_1}\in C^+_{\alpha_1}$}. Let $\bar{\alpha}>0$ be such that
\begin{equation}\label{eq:bar_alpha}
\|U_\alpha-U_{\alpha_1}\|_\infty < \frac{\alpha_1}{2},
\end{equation}
for all $\alpha\in(\alpha_1,\bar{\alpha})$.
Take $\alpha\in(\alpha_1,\bar{\alpha})$. Thanks to Proposition \ref{prop:linfty_stability}, we know that there exist $\delta,\mu\in\R$ such that $U_\alpha$ is affine in $[x+\delta,x_{\alpha_1}+\mu]$. We claim that $\mu>0$. Indeed, since $F+\alpha>F+\alpha_1$, we obtain that $U_\alpha$ cannot touch the graph of $F+\alpha$ in the interval $[x+\delta,x_{\alpha_1}+\mu]$. Moreover, thanks to \eqref{eq:bar_alpha}, we have that $U_\alpha>F-\alpha$ in the same interval. This proves the claim.

Let $x_\alpha>x$ be the largest value for which $U_\alpha$ is affine in $[x,x_\alpha]$.
we now have three cases to consider.

\emph{Case 3.1: $x_\alpha\in C^+_{\alpha}$.} With a similar argument to that used in case 2, with the role of $U_{\alpha_1}$ played by $U_\alpha$, and that of $U_{\alpha_2}$ played by $U_{\alpha_1}$, we conclude.

\emph{Case 3.2: $x_\alpha=1$.} In this case we use an argument similar to those of case 1.

\emph{Case 3.3: $x_\alpha\in C^-_\alpha$.} We claim that, up to taking $\alpha$ sufficiently close to $\alpha_1$, this is not possible. Let's argue by contradiction. Assume there exists $(\alpha_n)_{n\in\N}$ with $\alpha_n>\alpha_1$ and $\alpha_n\to\alpha$ when $n\to\infty$ such that $x_{\alpha_n}$, the largest value where $U_{\alpha_n}$ is affine in $[x,x_{\alpha_n}]$, is such that $x_{\alpha_n}\in C^-_{\alpha_n}$. Then, up to extracting a subsequence, we would have $x_{\alpha_n}\to x_0$. Since $U_{\alpha_n}\to U_\alpha$ uniformly, we must have $x_0=x_{\alpha_1}$. Therefore, using again the uniform convergence, we would get
\[
\lim_{n\to\infty} U_{\alpha_n}(x_n) = U_{\alpha_1}(x_{\alpha_1}).
\]
This is not possible, since
\[
U_{\alpha_n}(x_n) = F(\alpha_n)-\alpha_n > F(\alpha_n)+\alpha_n\to F(x_{\alpha_1})+\alpha_1 = U_{\alpha_1}+\alpha_1.
\]
Since $\alpha_1>0$, the strict inequality above does not saturate. This gives the desired contradiction and concludes the proof.
\end{proof}

Finally, we prove the monotonicity of the amplitude of the jump.

\begin{proposition}\label{prop:monot_ampl}
Let $\alpha_1 < \alpha_2$.
Then, it holds that
\[
|u_{\alpha_1}^l(x) - u_{\alpha_1}^r(x)| \geq  |u_{\alpha_2}^l(x) - u_{\alpha_2}^r(x)|,
\]
for all $x\in J_{u_{\alpha_2}}$.
\end{proposition}

\begin{proof}
Note that $x\in J_f$ if and only if $x$ is not a point of differentiability for $F$.
In particular, thanks to our assumption, if $F^{',-}(x)\neq F^{',+}(x)$, namely if the left and the right slopes of $F$ differ.

Fix $\alpha_1<\alpha_2$, and let $x\in J_{u_{\alpha_2}}$.
By Proposition \ref{prop:jump-inclusion}, we have two possibilities:
\[
x \in  J_f \cap \interior (C_{\alpha_2}^+ \cup C_{\alpha_2}^-),
\quad\quad\quad \text{or}\quad\quad\quad 
x\in J_f \cap \partial (C_{\alpha_2}^+ \cup C_{\alpha_2}^-).
\]
In the first case, assume without loss of generality that $x\in J_f\cap \interior C_{\alpha_2}^+$.
Then, there exists $\delta>0$ such that $(x-\delta,x+\delta)\subset \interior C_{\alpha_2}^+$. This implies that
\[
U_{\alpha_2}^{',-}(x) = F^{',-}(x),\quad\quad\quad
U_{\alpha_2}^{',+}(x) = F^{',+}(x).
\]
The monotonicity of the contact set with respect of $\alpha$ (see Theorem \ref{thm:contact_set_inclusion}) implies that $x \in   J_{f} \cap \interior C_{\alpha_1}^+$.
Therefore,
\[
U_{\alpha_2}^{',-}(x) = U_{\alpha_1}^{',-}(x),\quad\quad\quad
U_{\alpha_2}^{',+}(x) = U_{\alpha_1}^{',+}(x),
\]
and thanks to Theorem \ref{thm:gra07}, we get that $u^l_{\alpha_2}(x) = u^l_{\alpha_1}(x)$ and $u^r_{\alpha_2}(x) = u^r_{\alpha_1}(x)$.

In the second case, assume without loss of generality that $x\in J_f \cap \partial C_{\alpha_2}^-$.
We consider two sub-cases:
\[
x\in J_f \cap \interior C_{\alpha_1}^-\quad\quad\quad \text{or}
\quad\quad\quad x\in J_f \cap \partial C_{\alpha_1}^-.
\]
In both of these sub-cases, we assume without loss of generality that there exists $\delta>0$ such that $(x-\delta,x)\subset (0,1)\setminus (C^+_{\alpha_2}\cup C^-_{\alpha_2})$.

In the first sub-case, note that
\[
U_{\alpha_1}^{',+}(x) = f^+(x),\quad\quad\quad
f^-(x) = U_{\alpha_1}^{',-}(x)
\]
Then, (DL) of Theorem \ref{thm:optcond} yields that
\[
U_{\alpha_1}^{',+}(x) = f^+(x)
\leq U_{\alpha_2}^{',+}(x) \leq U_{\alpha_2}^{',-}(x)
\leq f^-(x) = U_{\alpha_1}^{',-}(x).
\]
Using Theorem \ref{thm:gra07}, we obtain the desired inequality.

In the second sub-case, we claim that
\[
U_{\alpha_2}^{',-}(x) \leq  U_{\alpha_1}^{',-}(x),
\]
and that
\[
U_{\alpha_1}^{',+}(x) \leq U_{\alpha_2}^{',+}(x).
\]
Let us start by proving the first claim.
Thanks to Theorem \ref{thm:contact_set_inclusion}, up to reducing $\delta>0$, we can assume that $(x-\delta,x)\subset (0,1)\setminus (C^+_{\alpha_1}\cup C^-_{\alpha_1})$.
Recall that in Step 3 of the proof of Theorem \ref{thm:contact_set_inclusion}, we proved that
\[
U_{\alpha_1}(x-\delta) - (\alpha_2-\alpha_1) \leq U_{\alpha_2}(x-\delta).
\]
Therefore, since $U_{\alpha_1}$ and $U_{\alpha_2}$ are affine in $(x-\delta,x)$, and $U_{\alpha_1}(x)= U_{\alpha_2}(x)$, we obtain the desired result.

To prove the second claim, we notice that if we are able to reduce $\delta$ so that $(x,x+\delta)\subset C^-_{\alpha_1}$, then we conclude by using (DL) of Theorem \ref{thm:optcond}.
Otherwise, we argue as in the proof of the previous claim.
\end{proof}

\section*{Acknowledgements}
The research of Riccardo Cristoferi was partially supported under NWO-OCENW.M.21.336.
Rita Ferreira was partially supported by King Abdullah University of Science and Technology (KAUST) baseline funds and KAUST
OSR-CRG2021-4674.
The research of Irene Fonseca was partially supported by the National Science Foundation under grants DMS-2108784, DMS-2205627, and DMS-2342349.
Part of this work was carried out while Riccardo Cristoferi and Jos\'e A. Iglesias attended the Workshop Calculus of Variations NL 2024 in Schiermonnikoog, for which those authors are grateful.
All authors also thank the hospitality of the Erwin Schr\"odinger International Institute for Mathematics and Physics (ESI) in Vienna, where part of this work was carried out.

\bibliographystyle{plain}

\IfFileExists{"CrFeFoIg-bib.bib"}
{\bibliography{CrFeFoIg-bib}}

\end{document}